\documentclass[a4paper]{article}
\usepackage[utf8]{inputenc}
\usepackage{amsmath,amssymb,amsbsy,dsfont,amsthm}

\usepackage[a4paper]{geometry}
\usepackage{xcolor}
\usepackage{algorithm}
\usepackage{algorithmic}
\usepackage{graphicx}
\usepackage[sort&compress,square,comma,authoryear]{natbib}
\usepackage{booktabs,multirow}
\usepackage{authblk}
\usepackage{hyperref}

\let\ifposter\iffalse

\usepackage{xfrac}
\usepackage{placeins}

\newcommand\p[1]{{\left(#1\right)}}
\newcommand\abs[1]{{\left|#1\right|}}
\newcommand\braces[1]{{\left\{#1\right\}}}
\newcommand\brackets[1]{{\left[#1\right]}}
\newcommand\cond{\,\middle|\,}

\newcommand\obs{y}
\newcommand\nobs{N}
\newcommand\obsspace{\mathcal{Y}}

\newcommand\lat{Z}
\newcommand\latspace{\mathcal{Z}}
\newcommand\para{\theta}
\newcommand\paraspace{\Theta}
\newcommand\marg{g}
\newcommand\nlf{h}
\newcommand\comp{f}
\newcommand\post{p}
\newcommand\kernel{\Pi}
\newcommand\prop{q}
\newcommand\iter{k}
\newcommand\pas{\gamma}
\newcommand\mle{\hat{\theta}}

\newcommand\estFIM[1]{\widehat I\p{#1}}
\newcommand\estFIMwhole[1]{\widehat I_{\text{whole}}\p{#1}}
\newcommand\expectation{\operatorname{E}}

\newcommand\ouv[2]{\p{#1,#2}}
\newcommand\norm[1]{{\left\|#1\right\|}}

\ifposter
\else
\theoremstyle{plain}
\newtheorem{theorem}{Theorem}[section]

\newtheorem{lemma}[theorem]{Lemma}

\theoremstyle{definition}

\newtheorem{assumption}[theorem]{Assumption}
\theoremstyle{remark}
\newtheorem{remark}[theorem]{Remark}
\fi

\title{Efficient preconditioned stochastic gradient descent for estimation in  latent variable models}
\author[1]{Charlotte Baey}
\author[2]{Maud Delattre}
\author[1]{Estelle Kuhn}
\author[3]{Jean-Benoist Leger}
\author[4]{Sarah Lemler}
\affil[1]{Univ. Lille, CNRS, UMR 8524 - Laboratoire Paul Painlevé, F-59000 Lille, France}
\affil[2]{Université Paris-Saclay, INRAE, MaIAGE, 78350, Jouy-en-Josas, France.}
\affil[3]{Université de technologie de Compiègne, CNRS, Heudiasyc, Compiègne, France}
\affil[4]{Université Paris-Saclay, CentraleSupélec, Mathématiques et Informatique pour la Complexité et les Systèmes, 91190, Gif-sur-Yvette, France.}

\date{June 2023}

\begin{document}
\maketitle

\begin{abstract}
	Latent variable models are powerful tools for modeling complex phenomena involving in particular partially observed data, unobserved variables or underlying complex unknown structures. Inference is often difficult due to the latent structure of the model. To deal with parameter estimation in the presence of latent variables, well-known efficient methods exist, such as gradient-based and EM-type algorithms, but with practical and theoretical limitations.  In this paper, we propose  as an alternative for parameter estimation an efficient preconditioned stochastic gradient algorithm. Our method includes  a preconditioning  step based on a positive definite  Fisher information matrix estimate. We prove convergence results for  the proposed algorithm under mild assumptions for very general latent variables models. We illustrate through relevant simulations  the performance of the proposed methodology in a nonlinear mixed effects model and in a stochastic block model.
\end{abstract}

\section{Introduction}

Latent variable models are widely used in many fields to describe complex phenomena whose mechanisms are indirectly observed and whose consideration in the model requires the use of unobserved variables. One can mention, for instance, mixture models \cite{mclachlan1988mixture} or stochastic block models \cite{abbe2018,lee2019} that are respectively used to describe the existence of an unknown group structure in a population and in an interaction network, or mixed-effects models whose latent structure is intended to describe some inter-individual variability \cite{pinheiro2000,lavielle2014}. While bringing a certain level of detail to the modeling, the use of latent variables leads to more complex inference in general, as the observed likelihood often does not have an explicit form. This has led to the development of specific numerical methods for parameter estimation in latent variable models. Among the most common approaches, one can find the EM algorithm \cite{dempster1977} and its variants such as the Stochastic Approximation EM (SAEM) algorithm \cite{delyon1999} or even the Variational EM (VEM) algorithm \cite{bernardo2003variational}. There are also gradient-based methods (see \textit{e.g} chapters 10-11 in \citet{cappe2005}), in particular some stochastic versions of the gradient descent algorithm have been specifically developed for latent variable models \cite{gu1998stochastic, cai2010metropolis, fang2021estimation}. 

EM-type algorithms are popular for their ease of implementation in curved exponential models where only operations on the sufficient statistics of the model are required at each iteration. These algorithms can still be applied in more general latent variable models but the methodology is not generic and requires new developments for each new model considered. To face this restrictive assumption, \citet{debavelaere2021} suggested to use the exponentialization trick which consists in performing inference in an extended model belonging to the curved exponential family instead of in the initial model. However  this approach has limitations in practice, mainly due to difficult algorithmic settings and tuning. Besides, the theoretical properties of EM-type algorithms have been the subject of many contributions. To our knowledge, the existing convergence theorems however all assume that the model belongs to the curved exponential family and none of them brings any guarantee beyond this framework (see \citet{wu1983} for EM and \citet{delyon1999} for SAEM). 

Gradient-based methods are an attractive alternative when EM-type algorithms are not easy to implement since their main ingredient is the log-likelihood gradient which is easily available in any latent variable model or at least approximated by combining the Fisher's identity (see \citet{cappe2005}) with Monte-Carlo methods. Among these, we find in particular the algorithms of \citet{cai2010metropolis} and \citet{fang2021estimation}. The former has no theoretical guarantee and the latter, although supported by a convergence theorem, applies only to models that do not contain parameters to be estimated in the distribution of the latent variables. In recent stochastic gradient algorithms, some attention is also given to variance control for the full gradient estimation when it is obtained by processing mini-batches of the original sample (see \textit{e.g.} \cite{johnson2013accelerating, reddi2016stochastic, fang2018spider}). For all the above mentioned algorithms, although often supported by theoretical convergence guarantees, the practical performance is strongly affected by the choice of the learning rate which is a difficult setting. To overcome this difficulty and speed up the convergence, some people propose to precondition the gradient, but as suggested in \citet{li2017preconditioned}, one also has to be careful with the definition of the preconditioner so as not to degrade the behavior of the algorithm.

In this paper, our contribution consists in presenting a new stochastic gradient algorithm for maximum likelihood parameter estimation in general latent variable models.  The proposed algorithm distinguishes itself from other stochastic gradient-based algorithms by including an easily available and structurally positive definite preconditioner based on \citet{delattre2019}, which is also an estimate of the Fisher information matrix (FIM) in independent data models. As the FIM corresponds to the Hessian matrix of the objective function, it is therefore a natural choice regarding the second-order approximation (see Li, 2017). In addition, as the FIM estimate proposed by \citet{delattre2019} has the nice structural property of being symmetric positive definite, our  preconditioning step allows to scale the different directions of the parameter space, homogenizing the evolution of the algorithm, and ensures that the search direction corresponds to a descent direction.

Since the algorithm entails updating the estimate of the Fisher information matrix at each iteration, asymptotic confidence intervals can also be easily computed for all parameters as a by-product of the algorithm.
Theoretical convergence results are provided that, unlike the existing results for competing EM-type algorithms, are also valid beyond the curved exponential family. The algorithm implementation is straightforward in practice, as it only requires the computation of the gradients of the log-densities of the latent variables and the gradients of the conditional log-densities of the observations given the latent, both of which being readily available quantities for the classical stochastic gradient descent. In addition, we propose a generic warming procedure which simplifies tuning and improves the algorithm efficiency in practice. Finally, it can also be easily extended  to Bayesian maximum a posteriori estimation as well as to regularized estimation.

The paper is organized as follows. Section~\ref{sec:mle_latentvarmodel} introduces latent variable models. Section~\ref{sec:algo_fsgd} presents our new Fisher-preconditioned stochastic gradient algorithm called Fisher-SGD and provides theoretical analysis. Some details on algorithmic settings are also given. Numerical results are presented in Section~\ref{sec:numerical} that show the good performances of the algorithm. Some concluding remarks are given in Section~\ref{sec:conclusion}.

\section{Maximum Likelihood Estimation in latent variable models}\label{sec:mle_latentvarmodel}

\subsection{Description of latent variable models}

Let us consider observed random variables denoted by $Y$ taking value in $\obsspace$  and latent random variables denoted by $\lat$ taking value in $\latspace$ which are not observed.
We assume that the couple $(Y,\lat)$ admits  a parametric density $\comp$ parameterized by $\para$ taking value in $\paraspace\subset \mathbb{R}^d$, where $d$ is a non-zero positive integer. We denote by $y$ and $z$ realizations of the random variables $Y$ and $Z$ respectively. We denote by $\post_\para \p{\cdot \cond y}$ the density of the posterior distribution, i.e. the conditional distribution of $Z$ given $y$.

Popular examples are mixture models, mixed effects models \cite{davidian1995}, hidden Markov models \cite{cappe2005}, stochastic block models \citet{nowicki2001estimation}, or frailty models \cite{duchateau2008}.

Estimation of  model parameters is not trivial in these models due to the presence of the latent structure  and the unobserved variables $\lat$. Namely one has to estimate $\para$ only using the observed values of variables $Y$ denoted by $\obs$.

\subsection{Examples}\label{sec:examples}
In this section, we provide several examples of latent variable models, that will be used in the numerical experiments.

\subsubsection{Mixed-effect models}\label{subsec:mixedeffects}
Mixed-effects models are commonly used when repeated data are available for each observational unit, e.g. in longitudinal studies or population models. They allow to account for both intra- and inter-individual variabilities through the use of fixed and random effects, the former being common to all the individuals while the latter vary from one individual to the other.
These models can be described hierarchically, with a first layer giving the marginal distribution of the latent variables $\lat$, and a second layer specifying the conditional distribution of the observations $Y$ given the latent variables $\lat$. More specifically, denoting by $Y_{ij}$ the $j$-th observation of individual $i$, with $j=1,\dots,J$ and $i=1,\dots,\nobs$, we consider the following model:
\begin{equation*}
	\left\{
	\begin{array}{ll}
		\lat_i \sim \mathcal{N}(\beta,\Gamma) & \\
		Y_{ij} = \nlf(\alpha,\lat_i,X_{ij}) + \varepsilon_{ij}, & \varepsilon_{ij} \sim \mathcal{N}(0,\sigma^2)
	\end{array}
	\right.
\end{equation*}
where $\alpha$ is a vector of fixed effects, $X_{ij}$ is a set of known covariates, $\nlf$ is a nonlinear function representing the intra-individual variability and $\varepsilon_{ij}$ is a random error term. The random effects are the latent variables $\lat_i$, and are assumed to be mutually independent. The sequences $(\lat_i)$ and $(\varepsilon_i)$, with $\varepsilon_i = (\varepsilon_{i1},\dots,\varepsilon_{iJ})^T$ are also assumed mutually independent.
The parameters to be estimated are $\alpha,\beta, \Gamma$ and $\sigma^2$.

We consider now the specific case of the logistic growth curve model, which is commonly used in the nonlinear mixed-effect models community (see e.g. \citet{Pin00} and their famous example of orange trees growth), where function $\nlf$ is given by:
\begin{equation}\label{eq:orangetrees}
	\nlf(\alpha,\lat_i,X_{ij}) = \frac{\lat_{i1}}{1 + \exp\p{-\frac{X_{ij} - \lat_{i2}}{\alpha}}},
\end{equation}
where $\lat_i = \p{\lat_{i1},\lat_{i2}}^T \sim \mathcal{N}(\beta, \Gamma)$, with $\beta = \p{\beta_1,\beta_2}^T$ and $\Gamma$ a $2\times2$ symmetric positive definite matrix. The model parameters are $\beta \in \mathbb{R}_+^* \times \mathbb R$, $\alpha \in \mathbb R^*_+$, $\Gamma \in \mathbb S_2^{++}$, where $\mathbb S^{++}_p$ is the set of symmetric, positive definite matrices of size $p \times p$ and $\sigma \in \mathbb R ^+$.

Due to the presence of the fixed effect $\alpha$, the joint density of $(Y_i,\lat_i)$ does not belong to the curved exponential family as defined in \citet{delyon1999}. Note that  the implementation of stochastic versions of the EM algorithm are not trivial in such cases, involving complex  terms evaluated by induction.

\subsubsection{Stochastic block models}\label{subsec:SBM}
The Stochastic Block Model (SBM) is a common model in graph analysis, introduced by \citet{holland1983stochastic} and \citet{nowicki2001estimation}. For a directed graph of $\nobs$ nodes, the SBM assumes a latent unknown node classification (here $Z$), and assumes edge presences as independent and identically distributed conditionally on the cluster of nodes with a probability distribution which depends only on node clusters. We note $Y$ the adjacency matrix of the graph, $K$ the number of groups, $Z_i$ the latent class indicator (one hot encoding) of node $i$. The SBM is formulated as follow:
\begin{equation}
	\left\{
	\begin{array}{ll}
		Z_i\sim\mathcal M\p{1; \alpha} & \text{ with } \alpha\in{\mathbb R_+^*}^K,\; \sum_k \alpha_k=1 \\
		Y_{ij}|Z_{ik}Z_{jl}=1\sim\mathcal{B}\p{p_{kl}} & \text{ with } \forall k,l,\; p_{kl}\in\ouv01 \\
	\end{array}
	\right.
\end{equation}

In this formulation, matrix $Y$ is the observed variable and $Z$ the latent variable. The parameters are $\alpha\in{\mathbb R_+^*}^K$ s.t. $\sum_k \alpha_k = 1$ and $p\in\ouv01^{K\times K}$.

Please note the non independence of $\p{Z_i}_{i:1\le i\le \nobs}$ conditionally on $Y$, thus the log-likelihood will be not splittable in terms involving only one $Z_i$ each.

\subsection{Maximum likelihood estimation in latent variable models}

We consider  the marginal density of $Y$ denoted by $\marg$ and defined by
\begin{equation*}
	\marg\p{\obs;\para}=\int_\latspace \comp\p{\obs,z;\para} \; dz
\end{equation*}

The maximum likelihood estimate (MLE) for parameter $\para$, denoted by $\mle$, is defined as:
\begin{equation*}
	\mle=\arg\max_\para  \marg\p{\obs;\para}
\end{equation*}

This estimate is very popular in statistics since it has very nice asymptotic properties in a wide class of statistical models. Assuming mild conditions, as the number of observations goes to infinity, the MLE is consistent, asymptotically Gaussian and efficient \cite{vandervaart2000}.  Therefore one can build asymptotic confidence intervals for model parameters as soon as an estimate of the Fisher information matrix is available.

However, computing this estimate in latent variable models through the maximization of the marginal likelihood $\marg$ often requires numerical tools. Indeed the marginal density does not usually admit an explicit expression. Most popular tools are EM-like algorithms \cite{ng2012}, used to maximize the density $\marg$ with respect to $\para$.

Another way to compute the MLE $\mle$ in latent variable models, often omitted but well discussed in \citet{cappe2005}, consists in searching the zeros of the derivative of $\log \marg$ using gradient-based methods. Indeed, assuming regularity conditions on $\comp$,  we get the Fisher identity \cite{cappe2005} which states that for all $\para \in \paraspace$:
\begin{equation}
	\label{eq:fisheridentity}
	\nabla_\para \log \marg\p{\obs;\para} = \expectation\p{ \nabla_\para \log \comp\p{\obs,Z;\para} \cond \obs ;\para}
\end{equation}

Solving a function defined as an expectation can be done using stochastic gradient algorithms. Note however that the quantity $\para$ is involved twice in the Fisher identity, namely in the derivative of the log-likelihood and in the posterior distribution of $\lat$.

Therefore we will consider in the following section a stochastic gradient type algorithm to solve the Fisher identity \eqref{eq:fisheridentity}.

\section{Efficient stochastic gradient algorithm with preconditioning step}\label{sec:algo_fsgd}

In this section, we present our algorithm Fisher-SGD and some convergence results. We first present the algorithm in the case of independent observations and give in a second time a more general version for non-independent observations.

\subsection{Description of the algorithm in the independent case}

One of the main ideas of our algorithm is to pre-condition, in the stochastic gradient descent at iteration $\iter$, by the positive definite estimate of the Fisher information $\estFIM{\para}$ proposed by \citet{delattre2019} and detailed in Appendix \ref{sec:app:proofs}. The algorithm is described in Algorithm~\ref{alg:1} (more details can be found in Algorithm~\ref{alg:detailled} in the Appendix).

\begin{algorithm}
	\caption{Fisher-SGD in the independent case}
	\label{alg:1}
	\begin{algorithmic} 
		\STATE{\bfseries Input:} $z_0,\para_0$, $y_1,\dots,y_N$,$r$
		\FOR{$\iter=1, \dots, K$}  
		\FOR{ $i=1,\dots, \nobs$}
		\STATE{$z_i^{\iter} \sim \prop $ where $\prop$ is either the posterior $\post_{\para_{\iter-1}} \p{ \cdot \cond y_i }$ or a Markov kernel $\kernel_{\para_{\iter-1}} \p{ \cdot \cond \cdot, y_i }$ }
		\STATE{ Compute $\Delta_i^{\iter}=\p{1-\pas_\iter}\Delta_i^{\iter-1}+\pas_\iter\nabla_\para\log\comp\p{y_i,z_i^\iter;\para_{\iter-1}}$}
		\ENDFOR
		\STATE{Compute $v_\iter=\frac{1}{\nobs}\sum_{i=1}^\nobs\nabla_\para \log\comp\p{y_i,z_i^\iter;\para_{\iter-1}}$}
		\STATE{Compute  $I_\iter=\dfrac{1}{N}\sum_{i=1}^N\Delta_i^{\iter}\p{\Delta_i^{\iter}}^T$}
		\STATE{Set $\para_\iter=\para_{\iter-1}+\pas_\iter
			I_\iter^{-1}v_\iter$}
		\ENDFOR
		\STATE{\bfseries Output:} Estimated parameter $\para_\iter$, estimated Fisher information matrix of the whole sample $ N I_\iter$, $r$-quasi-sample of latent variables in the posterior distribution $\p{z^{K-r+1}\cdots z^K}$.
	\end{algorithmic}
\end{algorithm}

\begin{remark} 
	\begin{enumerate}
		\item The proposed algorithm computes the MLE in latent variable models without requiring that the joint density  belongs to the  curved exponential family.
		\item The proposed algorithm allows to compute asymptotic confidence intervals for free since the MLE and a Fisher matrix estimate are available as a by-product of the algorithm.
		\item The proposed algorithm can be run either using exact simulation of the latent variables according to their conditional distribution given the observations whenever possible, or using the transition kernel of an ergodic Markov Chain having the posterior distribution as invariant distribution. Note that one common practical choice for such MCMC sampling scheme is the Metropolis-Hastings or the Metropolis-within-Gibbs algorithm \cite{robert1999}.
		\item The proposed algorithm can be extended to compute maximum a posteriori (MAP) estimate in Bayesian settings or regularized estimates, by adding to the criterion to be maximized a term corresponding to the prior distribution or the regularization term, and then considering an additional proximal step as in \citet{atchade2017perturbed}.
		\item The stochasticity of the stochastic gradient in our setting is due to the  sampling of the latent random variable which allows to compute the target criterion defined as an expectation on the latent space.
		\item In settings where the number of independent observations is large, usual minibatch techniques can be easily included in the proposed algorithm to deal with the high dimension of the observations (\citet{schmidt2017minimizing}).
	\end{enumerate}
\end{remark}

\subsection{Theoretical results}

We now present two convergence results for the algorithm, depending on how the realizations of $\lat$ are generated at each iteration of the algorithm, either from the posterior distribution  or from the  transition kernel of a Markov Chain Monte Carlo algorithm. For the sake of simplicity in the  next section we drop the index $1 \leq i \leq N$ in the notations. We first need some general regularity assumptions on the statistical model and on the sequence of step sizes $(\pas_\iter)$.

\begin{assumption}
	\label{ass:model}
	\begin{enumerate}
		\item \label{ass:joint} The joint density $f$ is twice differentiable for $\para \in \paraspace$.
		\item \label{ass:obs} For all $\obs \in \obsspace $ the observed log-likelihood $\log g\p{\obs,\para}$ is continuously differentiable on $\paraspace$ and 
		\begin{equation*}
			\nabla_\para g(\obs,\para)=\int_\mathcal{Z}\nabla_\para f(\obs,z,\para)dz.
		\end{equation*}
		\item \label{ass:stepsize} The sequence of step-sizes $\p{\pas_\iter}_k$ satisfies for all $\iter\geq 0$, $0\leq\pas_\iter\leq 1$, $\sum_{\iter=1}^{+\infty}\pas_\iter=+\infty$ and $\sum_{\iter=1}^{+\infty}\pas_k^2<+\infty$.
	\end{enumerate}
\end{assumption}

All these assumptions are classical for maximum likelihood estimation in latent variable models \cite{delyon1999, kuhn2004}.

We omit in the sequel the dependency of several quantities in $\obs$ since it is considered as fixed. Let us define the objective function to minimize $F\p{\para}=-\log g\p{\obs,\para}$. Therefore solving the Fisher identity (\ref{eq:fisheridentity}) is equivalent to solving $\nabla_\para F\p{\para}=0$.

\subsubsection{First case: simulating from the posterior}

To obtain the convergence of the algorithm in this context, we need  additional assumptions. 

\begin{assumption}
	\label{ass:nonconvex}
	\begin{enumerate}
		\item \label{ass:L} The gradient of $F$ is L-Lipschitz on $\paraspace$. 
		\item \label{ass:C} There exists $C>0$ such that for all $\obs$ and 
		$\para$, $\expectation\brackets{\norm{\nabla_\para \log f\p{\obs,\lat;\para}}^2}\leq C$.
		\item \label{ass:eigen} There exist $\mu_m>0$ and $\mu_M>0$ such that for all $k$, 
		$\forall \mu\in\operatorname{Eig}(I_\iter),\, \mu_m<\mu<\mu_M$, where
		$\operatorname{Eig}(A)$ denotes the set of eigenvalues (the spectrum) of matrix $A$.
	\end{enumerate}
\end{assumption}

The first two assumptions are classical when proving the convergence of a stochastic gradient algorithm. The third one is specific to our pre-conditioning using a positive definite estimate of the Fisher information. Note that in all regular models where the FIM is positive definite, this last assumption is satisfied for $\nobs$ and $\iter$ large enough since the FIM estimate proposed by \citet{delattre2019} is convergent when the sample size $\nobs$ goes to infinity and the algorithm is convergent when the number of iterations $k$ goes to infinity.

\begin{theorem}
	\label{th:prop_post}
	Under Assumptions~\ref{ass:model} and~\ref{ass:nonconvex}, the iterates $(\para_\iter)_\iter$ defined in Algorithm~\ref{alg:1} with $\prop$ equals to the posterior distribution $\post_{\theta_{k-1}}$ at iteration $k$ satisfy the convergence guarantee
	\begin{align*}
		\expectation \brackets{\underset{0\leq l\leq k}{\min} \norm{\nabla_\para F\p{\para_l}}^2} \leq&  \dfrac{2\p{F\p{\para_0}-\min F}}{2\mu_{m}\sum_{l=0}^\iter\pas_l}\\
		&+\dfrac{\mu_{M}^2CL\sum_{l=0}^\iter\pas_l^2}{2\mu_{m}\sum_{l=0}^\iter\pas_l}.
	\end{align*}
\end{theorem}

Proof: See Appendix~\ref{sec:app:proofs}.

The control bound in Theorem \ref{th:prop_post} goes well to $0$ when $k$ goes to infinity and is similar to those obtained by, e.g., \citet{Bottou} or \citet{Ghadimi} for the standard stochastic gradient algorithm.

\subsubsection{Second case: simulating from a  kernel}

In this context, we need  additional assumptions on the Markov chain. Let us introduce the following notations: for a measurable function $V:\latspace \rightarrow [1, +\infty) $, a measure $\mu$ on the $\sigma$-field of $\latspace$ and a function $f:\latspace \rightarrow \mathbb{R}$, we define $$\abs{f}_V= \sup_{z \in \latspace} \frac{\abs{f(z)}}{V(z)}, \ \ \ \norm{\mu}_V= \sup_{f, |f|_V\leq 1}\abs{\int f d\mu}. $$

\begin{assumption}
	\label{ass:MCMC}
	\begin{enumerate}
		\item \label{ass:kernel}  For all $\para \in \paraspace$, the transition kernel $\kernel_{\para}$ is a Markov kernel with invariant distribution the posterior $\post_\para$. 
		
		\item \label{ass:drift} There exist $0<\lambda<1$, $0<b$, $p \geq 2$ and a measurable function $W: \latspace \rightarrow [1, +\infty)$ such that $\sup_{\para \in \paraspace}\abs{\nabla_\para \log f(\cdot,\obs;\para)}_W< \infty$ and 
		$\sup_{\para \in \paraspace}\kernel_\para W^p \leq \lambda W^p +b$.
		
		\item \label{ass:ergo} For any $0<\nu<p$, there exist a constant $C$ and $0<\rho<1$ such that for any $z \in \latspace$, $\sup_{\para \in \paraspace}\norm{\Pi_\para^m(z,\cdot)-\post_\para }_{W^\nu} \leq C \rho^m W^\nu(z)$.
		
		\item \label{ass:lipschitz} There exists a constant $C$ such that for any $(\para,\para') \in \paraspace^2$,
		\begin{align*}
			\abs{\nabla_\para \log f(\cdot,\obs;\para) -  \nabla_\para \log f(\cdot,\obs;\para')}_W &\leq C \norm{\para-\para'} \\
			\abs{\post_\para -  \post_{\para'}}_W &\leq C \norm{\para-\para'} \\
			\sup_z   \frac{\norm{\kernel_\para(z,\cdot) -\kernel_{\para'}(z,\cdot)}}{W(z)} &\leq C \norm{\para-\para'}.
		\end{align*}
	\end{enumerate}
\end{assumption}

These assumptions are standard sufficient conditions in the literature to ensure the uniform ergodicity of the Markov chain with respect to $\para$ and the existence of a solution to the Poisson equation associated to function $\nabla_\para \log f(\cdot,\obs;\para)$ (see for example \citet{allassonniere2015, fort2011}). Concerning the logistic growth nonlinear mixed effects model defined in section~\ref{subsec:mixedeffects} using (\ref{eq:orangetrees}), the function $W$ can be chosen equal to $W(z)=1+\norm{z}^2 $.

\begin{assumption}
	\label{ass:stepsizeadd}
	The sequence of step sizes  $(\pas_k)$ satisfies $\sum \abs{\pas_{k+1}-\pas_{k}} <\infty $.
\end{assumption}
This step size assumption is classical when controlling stochastic approximation and is satisfied in particular when the step size sequence is decreasing.

\begin{theorem}
	\label{th:prop_kernel}
	Under Assumptions~\ref{ass:model}, \ref{ass:nonconvex},  \ref{ass:MCMC},  \ref{ass:stepsizeadd} and assuming that $\paraspace$ is bounded, the iterates $(\para_\iter)_\iter$ defined in Algorithm~\ref{alg:1} with $\prop$ equal to a transition kernel  $\kernel_{\theta_{k-1}}$ satisfy the convergence guarantee
	\begin{align*}
		\expectation\brackets{\underset{0\leq l\leq k}{\min}\norm{\nabla_\para F(\para_l)}^2} \leq&\; \dfrac{2(F(\para_0)-\min F)}{2\mu_{m}\sum_{l=0}^\iter\pas_l}\\
		&+\dfrac{\mu_{M}^2CL\sum_{l=0}^\iter\pas_l^2+C}{2\mu_{m}\sum_{l=0}^\iter\pas_l}.
	\end{align*}
\end{theorem}

Proof: See Appendix~\ref{sec:app:proofs}.

This result extends that of Theorem~\ref{th:prop_post} to the setting often used in practice when it is not possible to generate the latent variables from the posterior distribution. We obtain a similar bound as the one in Theorem \ref{th:prop_post} with an additional residual term that appears because of the MCMC procedure used to simulate $z^k$. 
The proof is postponed to the supplementary material and relies on technical tools involving Poisson equation solution for the control of Markov chain where the parameter evolves simultaneously. 

\subsection{Description of the algorithm in the non-independent case}

In the non-independent case, the algorithm should be adapted as the log-likelihood is not separable into terms involving only one $Z_i$. The proposed methodology uses the global criterion $f\p{y,z;\para_\iter}$. The adapted algorithm is provided in Algorithm~\ref{alg:nonindep}. There is no general formula to write the estimation of the Fisher information matrix in this case. See the adaptation for the SBM in Section~\ref{subsec:SBM} for an example.

\begin{algorithm}
	\caption{Fisher-SGD in the non-independent case}
	\label{alg:nonindep}
	\begin{algorithmic} 
		\STATE{\bfseries Input:} $z_0,\para_0$, $y$, $r$
		\FOR{\iter=1,\dots,K}  
		\STATE{$z^{\iter} \sim \prop $ where $\prop$ is either the posterior $\post_{\para_{\iter-1}} \p{ \cdot \cond y }$ or a Markov kernel $\kernel_{\para_{\iter-1}} \p{ \cdot \cond \cdot, y }$ }
		
		\STATE{Compute $v_\iter=\nabla_\para\log\comp \p{y,z^\iter;\para_{\iter-1}}$}
		\STATE{Compute $I_\iter$ with a custom method.}
		\STATE{Set $\para_\iter=\para_{\iter-1}+\pas_\iter I_\iter^{-1}v_\iter$}
		\IF{$k>K-r$}
		\STATE{Compute $\nabla_\para^2\log\comp\p{y,z^\iter;\para_{\iter}}$}
		\ENDIF
		\ENDFOR
		\STATE{\bfseries Output:} Estimated parameter: $\para_\iter$, estimated Fisher information matrix of the whole sample: $\frac1r\sum_{k=K-r+1}^K\nabla_\para^2\log\comp\p{y,z^\iter;\para_{k}}$, sample of latent in the posterior distribution: $\p{z^{K-r+1}\cdots z^K}$.
	\end{algorithmic}
\end{algorithm}

\begin{remark}
	Theorems \ref{th:prop_post} and \ref{th:prop_kernel} can be immediately extended to the non-independent setting. 
\end{remark}

\subsection{Practical implementation}\label{sec:practical}

The implementation of the algorithm should be done in practice with some precautions, in
particular at the beginning of the algorithm for the learning step $\pas_k$
and the calculation of the preconditioning matrix $I_\iter$.

\subsubsection{Evolution of the learning step}

The authors propose a strategy in three steps, similar to the warm-up strategies used for example by \citet{loshchilov2016sgdr} and \citet{smith2019super}:

\begin{description}
	\item[pre-heating:] first, at the very beginning of the algorithm, the learning step is
	gradually increased following an exponential growth, starting from a very
	small value until reaching $1$.
	\item[heating:] then, the step is kept at $1$ during a certain time, constituting
	the heating period.
	\item[decreasing steps:] after stabilization, the steps are decreasing (such that $\sum_k
	\gamma_k=+\infty$ and $\sum_k \gamma_k^2<+\infty$).
\end{description}

We thus obtain:
\begin{equation*}
	\pas_\iter=
	\begin{cases}
		\pas_0^\p{1-\frac{\iter}{K_{\text{pre-heating}}}} &\text{ if } \iter \le K_{\text{pre-heating}}  \\
		1 &\text{ else if } \iter\le K_{\text{heating}}\\
		\p{\iter-K_{\text{heating}}}^{-\alpha} &\text{ else}
	\end{cases}
\end{equation*}

The authors propose to use $K_{\text{pre-heating}}=1000$ and $\pas_0=10^{-4}$. The proposed setting seems to be conservative for most situations, but $K_{\text{pre-heating}}$ and $\pas_0$ can be respectively increased and decreased to improve stabilization of the pre-heating phase.

Concerning the choice of $K_{\text{heating}}$, the authors propose to use an
adaptive method, averaging the norms of the gradients calculated with a third
order filter of constant $\frac1{1000}$ and to stop the heating phase when the norm of
the averaged gradient does not decrease anymore.

Concerning the choice of $\alpha$, to ensure $\sum_k
\gamma_k=+\infty$ and $\sum_k \gamma_k^2<+\infty$, the authors propose to use $\alpha=\sfrac23$. 

The complete algorithm using the
preheating and the adaptive length heating is given in the Appendix in Algorithm~\ref{alg:detailled}.

\subsubsection{Preconditioning matrix}

At the beginning of the algorithm, a bad preconditioning can lead to unstable
behavior. Thus during the whole preheating period, instead of using $I_\iter$
as defined in algorithm~\ref{alg:1}, we use a stabilized version. Let us introduce
$$ \widehat I^*_\iter
= \frac{1}{N}\sum_{i=1}^N\Delta_i^{\iter}\p{\Delta_i^{\iter}}^T $$
and define the new preconditioner by:
\begin{align*}
	I_\iter=&\begin{cases}
		\p{1-\pas_\iter}r_\iter Id+\pas_\iter\widehat I^*_\iter & \text{ if } \iter<K_{\text{pre-heating}}\\
		\widehat I^*_\iter &\text{ otherwise}.
	\end{cases}
\end{align*}

We can choose $r_\iter=1$ or, to further avoid
instabilities we can choose $r_\iter = \max\p{1,
	\operatorname{tr}\p{\widehat I^*_\iter}}$.

The complete algorithm using this stabilization is given in the Appendix in Algorithm~\ref{alg:detailled}.

\subsubsection{Auto-differentiation and parametrization}
Nowadays, there is high interest in using automatic differentiation over analytical calculation of gradients. The authors propose to compute all gradients by means of automatic
differentiation.

The parameters of the statistical models are rarely all parameters in $\mathbb
R^d$, but the algorithm is presented in the framework of $\paraspace=\mathbb
R^d$. Using the functional invariance property of maximum likelihood, the
authors propose to systematically reparametere to $\mathbb R^d$, through a
bijective reparametrization. For obvious reasons of derivatives use, these
reparametrizations must be diffeomorphisms and must be differentiable by
automatic differentiation to be used transparently in the framework proposed
here.  The parametrization Cookbook \cite{paramcookbook} introduces a set of classical
reparametrizations in statistics that verify these properties. The obtained
maximum likelihood properties (and, in particular, the confidence intervals) can
then be transferred to the initial space using a delta method \cite{vandervaart2000}.

\section{Numerical experiments}\label{sec:numerical}

\subsection{Logistic growth mixed-effects model}
We generated data according to the logistic growth model presented in Section~\ref{sec:examples}, with $\nobs=1000$ individuals, with the same vector of observation times for each individual, defined as a vector of $m=20$ equally spaced values between 100 and 1500, and using the following parameter values: $\beta = \p{200,500}^T, \Gamma = \operatorname{diag}\p{40,100}$, $\alpha = 150$ and $\sigma = 10$. 
Using a reparametrization, Algorithm~\ref{alg:1} was run in the reparameterized space $\mathbb{R}^d$ with $d=7$. The code is available in the Git repository \url{https://github.com/baeyc/fisher-sgd-nlme}.

To evaluate the performance and the robustness of our approach with respect to maximum likelihood estimate and confidence regions, we compared our algorithm to the competing MCMC-SAEM algorithm \cite{kuhn2004}. This algorithm is, to the best of our knowledge, the only other one providing theoretical convergence guarantees towards the MLE when the model belongs to the curved exponential family. Since we introduced a fixed effect in the model, the considered model does no longer belong to the exponential family. Therefore, we used a specific implementation of the MCMC-SAEM algorithm \citep{saemix} which rely on an exponentialization trick for models that do not belong to the curved exponential family, but is only usable for some specific nonlinear mixed-effects models. It is also noteworthy to mention that this algorithm relies on a block-diagonal estimate of the FIM, which has no particular reason to be block-diagonal in general.

The Fisher-SGD algorithm (we used the version given in Algorithm~\ref{alg:1} using a Metropolis-within-Gibbs sampler) was run on $M$=1000 datasets generated using the same parameter values. The competing MCMC-SAEM algorithm was run on the same simulated datasets using the R package \texttt{saemix}. 

Table~\ref{tab:emp_cov_logistic} gives the root mean squared errors (RMSE) associated with each parameter and the empirical coverages computed as the proportion of the $M$ datasets for which the true parameter value used to generate the data felt into the $95\%$ confidence region. The confidence regions were built using automatic differentiation and the delta method (see Appendix~\ref{sec:app:reparam} for details).
Our approach performed better than the MCMC-SAEM algorithm, especially for variance components parameters, i.e. for $\Gamma_{11}$, $\Gamma_{12}$ and $\Gamma_{22}$. Since the RMSE are similar with both algorithms, these results suggest that the FIM estimate obtained with Fisher-SGD is more accurate than the one obtained with the \texttt{saemix} package.

\begin{table}[t]
	\begin{center}
		\caption{Root mean squared error (RMSE) and empirical coverage of confidence regions built at the nominal level of 0.95 using the FIM and the parameter estimates, for a total of $M=1000$ repetitions. The first line corresponds to the vector of all parameters $\para$, and thus the coverage is associated to the confidence region in $\mathbb{R}^7$. The simulated values for the parameters are $\beta_1=200, \beta_2=500$, $\alpha = 150$, $\Gamma_{11}=40$, $\Gamma_{12} = 0$, $\Gamma_{22} = 100$ and $\sigma^2=100$.    }
		\label{tab:emp_cov_logistic}
		\vskip 0.15in
		\begin{small}
			\begin{sc}
				\begin{tabular*}{\columnwidth}{ccccc}
					\toprule
					\multirow{2}{*}{Type} & \multicolumn{2}{c}{Fisher-SGD} & \multicolumn{2}{c}{MCMC-SAEM} \\
					& RMSE & Coverage & RMSE & Coverage \\
					\midrule
					$\para$ & 15.13 & 0.952 $\pm$ 0.014 & 17.24 & 0.935 $\pm$ 0.015 \\
					$\beta_1$ & 0.234 & 0.942 $\pm$ 0.012 & 0.236 & 0.941 $\pm$ 0.015 \\
					$\beta_2$ & 0.586 & 0.958 $\pm$ 0.010 & 0.625 & 0.941 $\pm$ 0.015\\
					$\alpha$ & 0.414 & 0.972 $\pm$ 0.013 & 0.416 & 0.968 $\pm$ 0.011\\
					$\Gamma_{11}$ & 2.221 & 0.951 $\pm$ 0.013 & 2.241 & 0.949 $\pm$ 0.014 \\
					$\Gamma_{12}$ & 4.156 & 0.948 $\pm$ 0.014 & 4.334 & 0.935 $\pm$ 0.015\\
					$\Gamma_{22}$ & 14.324 & 0.948 $\pm$ 0.014 & 16.492 & 0.905 $\pm$ 0.018 \\
					$\sigma^2$ & 1.005 & 0.957 $\pm$ 0.012 & 1.010 & 0.951 $\pm$ 0.013\\
					\bottomrule
				\end{tabular*}
			\end{sc}
		\end{small}
		\vskip -0.1in
	\end{center}
\end{table}

As an illustration, Figure~\ref{fig:evol_theta} gives the evolution of one of the $M$ trajectories, in the original parameter space. We can see that the algorithm reaches a neighborhood of the true value at the end of the pre-heating phase, stabilizes itself around this true value during the heating phase and reaches convergence during the last phase. Figure~\ref{fig:evol_fim} in Appendix~\ref{sec:app_logistic_model} gives the evolution of the diagonal of the estimated FIM, along with the evolution of the learning step across the iterations.

\begin{figure}[bt]
	\centering
	\includegraphics[width=\columnwidth]{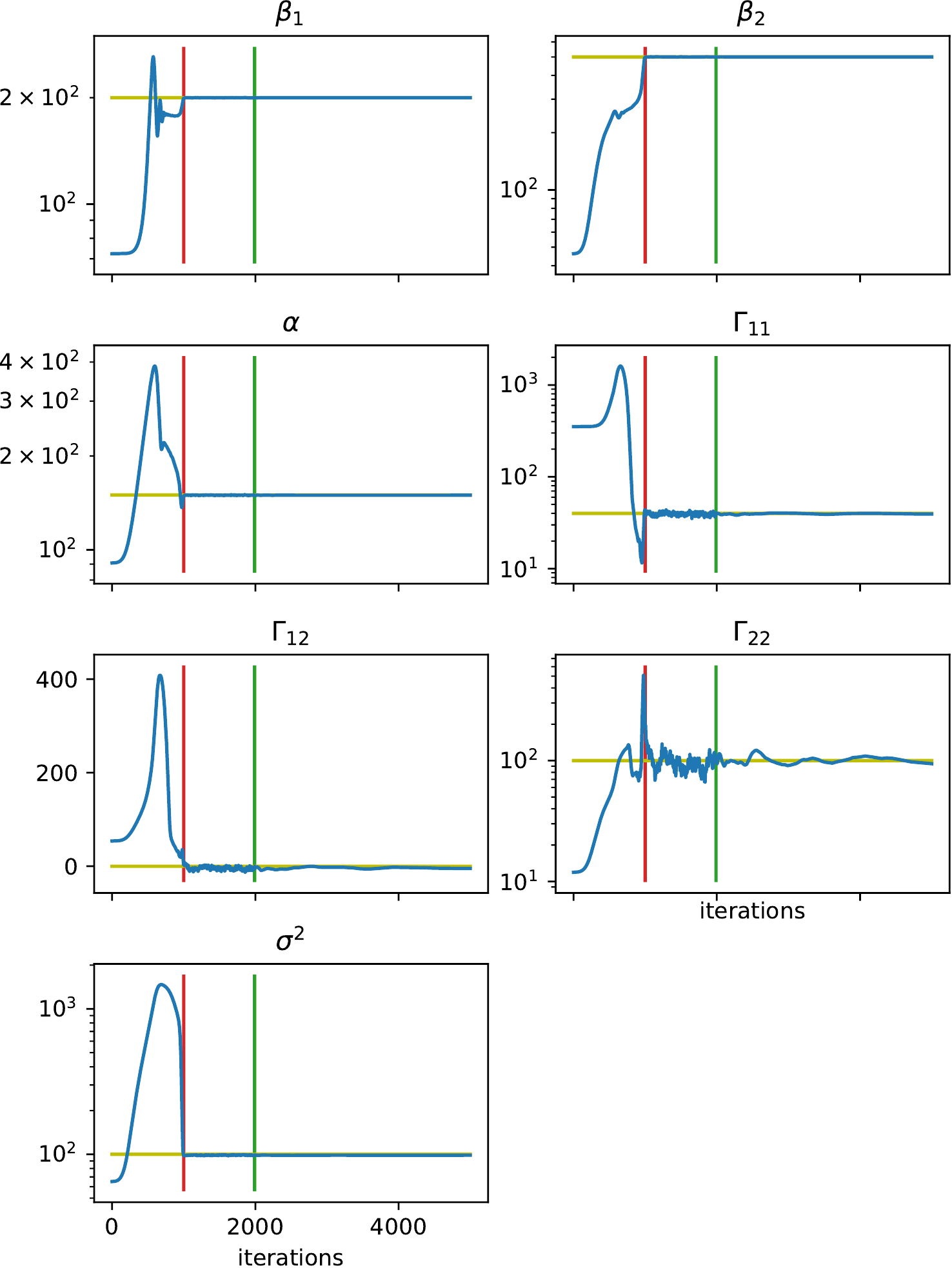}
	\caption{Evolution of the parameter estimates across the iterations in the logistic growth model. Yellow line: simulated value. The red line is the end of the pre-heating, and the green line is the end of the heating. }
	\label{fig:evol_theta}
\end{figure}

\subsection{Stochastic Block model}

We generated $2000$ simulated networks according to the Stochastic Block Model presented above in Section~\ref{sec:examples}, with $K=4$ groups and with $\nobs=100$ nodes or $\nobs=200$ nodes.

All networks are generated from the same set of parameters, we choose:
\begin{align*}
	\alpha&=\p{\sfrac14,\sfrac14,\sfrac14,\sfrac14}\\
	p&=\begin{bmatrix}
		\sfrac23&\sfrac23&\sfrac13&\sfrac23\\
		\sfrac23&\sfrac23&\sfrac23&\sfrac13\\
		\sfrac13&\sfrac23&\sfrac23&\sfrac23\\
		\sfrac23&\sfrac13&\sfrac23&\sfrac23\\
	\end{bmatrix}
\end{align*}

As $\alpha^Tp$ and $p\alpha$ are constant vectors, expected inner and outer degrees of node do not depend on clusters, therefore this simulation setting is not an easy case where naive algorithm can be applied, \textit{e.g.} \citet{channarond2012classification}.

As described in Section~\ref{sec:examples}, parameters are in a constrained space. To handle these constraints, a reparametrization is applied (see Appendix~\ref{sec:app:reparam} for details). The complete algorithm used for estimation in Stochastic Block Model is a particular case of Algorithm~\ref{alg:nonindep} where the preconditionning matrix is computed with the same idea than for the independent case and the sampling is made with a Gibbs sampler. The complete algorithm is given in Algorithm~\ref{alg:SBM} in the appendix. However, this algorithm is not sufficient to handle all the cases and for some initializations the sampling and the algorithm behavior leads to empty classes: in these cases the algorithm is restarted from the beginning with a new random initialization. The code is available in the Git repository \url{https://gitlab.com/jbleger/sbm\_with\_fisher-sgd}.

As other model-based classification methods, SBM is subject to label-switching and the parameter is identifiable up to a permutation of classes. To compare the estimation to the simulated value, classes are permuted to maximize the congruence between the posterior of $Z$ and the simulated value of $Z$.

The main advantage of our method is that we obtain an estimate of the FIM. As we have the asymptotic normality property \cite{bickel2013asymptotic}, we can compute asymptotic confidence interval of parameters. Then we illustrate our method by evaluating the estimation error with root mean squared error (RMSE) and by evaluating the coverage of 95\% confidence interval obtained with the parameter estimate and the Fisher Information Matrix estimate. See Appendix~\ref{sec:app:reparam} for details.

We chose not to compare ourselves to other methods, since to the authors' knowledge there is no method computing the MLE for not small SBM networks, disallowing computation of asymptotic confidence intervals.

\begin{figure}[ht]
	\centering
	\includegraphics[width=\columnwidth]{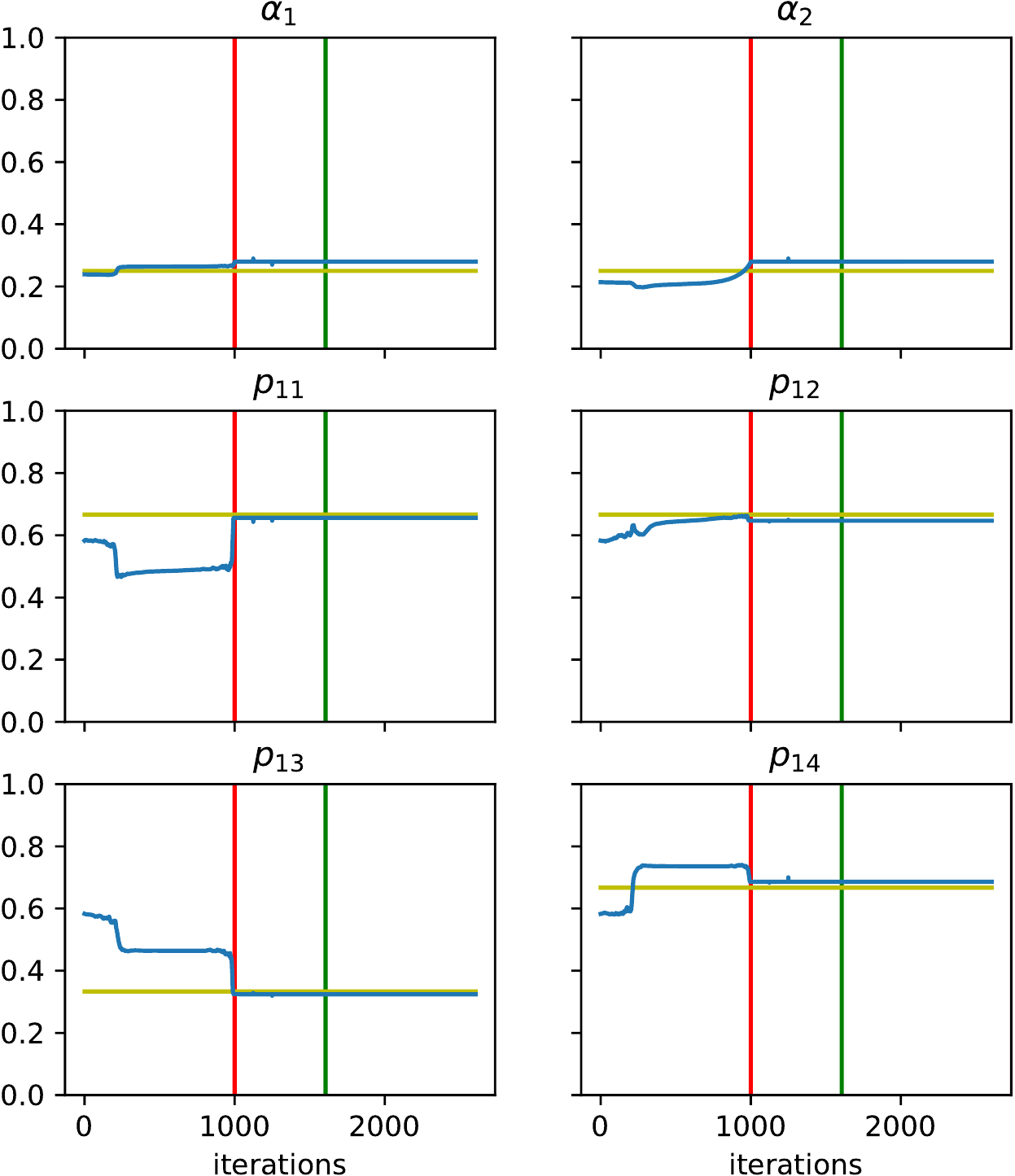}
	\caption{Evolution of the parameter estimates across the iterations with $\nobs=100$ in the stochastic block model. Yellow line: simulated value. The red line is the end of the pre-heating, and the green line is the end of the heating. Results for all parameters are given in appendix Figures~\ref{fig:sbm_all_alpha} and~\ref{fig:sbm_all_pi}.}
	\label{fig:sbm_main}
\end{figure}

In Figure~\ref{fig:sbm_main}, we show the evolution of parameter estimates (after transformation in the original parameter space). We see here the practical importance of the pre-heating phase: when latent variable are not acceptable, rapid evolution of parameters can lead to non convergent algorithm.

\begin{table}[ht]
	\begin{center}
		\caption{Results of Fisher-SGD applied on SBM with $2000$ replications. Root mean squared error (RMSE) and empirical coverage of confidence regions built at the nominal level of 0.95 using the FIM and the parameter estimates, for a total of $M=2000$ repetitions. The first line corresponds to the vector of parameters $\para$, and thus the coverage is associated to the confidence region in $\mathbb{R}^{K^2+K-1}$. See Table~\ref{tab:emp_sbm_details} in Appendix for complete results with $\nobs=100$ and $\nobs=200$.}
		\label{tab:emp_sbm}
		\vskip 0.15in
		\begin{small}
			\begin{sc}
				\begin{tabular*}{\columnwidth}{cccc}
					\toprule
					\multirow{2}{*}{Parameter}&\multirow{2}{*}{Simulated}& \multicolumn{2}{c}{$\nobs=100$} \\
					&& RMSE & Coverage \\
					\midrule
					$\theta$ &         & $0.648$ & $0.936\pm0.011$\\
					$\alpha_1$      & $0.250$ & $0.044$ & $0.943\pm0.010$\\
					$\alpha_2$      & $0.250$ & $0.044$ & $0.939\pm0.010$\\
					$p_{1,1}$       & $0.667$ & $0.023$ & $0.940\pm0.010$\\
					$p_{1,2}$       & $0.667$ & $0.019$ & $0.947\pm0.010$\\
					$p_{1,3}$       & $0.333$ & $0.022$ & $0.948\pm0.010$\\
					$p_{1,4}$       & $0.667$ & $0.019$ & $0.947\pm0.010$\\
					\bottomrule
				\end{tabular*}
			\end{sc}
		\end{small}
	\end{center}
	\vskip -0.1in
\end{table}

Results are presented for $\nobs=100$ for a subset of the original parameters in Table~\ref{tab:emp_sbm}. Results for all parameters for $\nobs=100$ and $\nobs=200$ are given in the appendix in Table~\ref{tab:emp_sbm_details}. We deduce that computed confidence ellipsoid for $\theta$ and confidence intervals for $\alpha$ and $p$ are correct, which validate the MLE and Fisher Information Matrix estimation provided by our algorithm.

\section{Conclusion}\label{sec:conclusion}

In this article, we consider parameter estimation in latent variable models. We propose an efficient stochastic gradient algorithm that includes  a preconditioning  step to scale the different directions of the parameter space, homogenizing the evolution of the algorithm. The preconditioner we use corresponds to a positive definite estimate of the Fisher information matrix in independent data models, which allows to get for free asymptotic confidence intervals for the parameters as a by-product of the algorithm. Theoretical convergence results are provided under mild assumptions for very general latent variables models, without assuming that the density  belongs to the curved exponential density family. Using simulations, we show that our new algorithm performs satisfactorily and gives similar to better performances compare to competing methods. As we also propose a warming procedure, the method is generic enough to be easily implemented in very general latent variable models.

\section*{Funding}
This work was funded by the Stat4Plant project ANR-20-CE45-0012.

\bibliography{calimero_biblio}
\bibliographystyle{plainnat}

\newpage
\appendix

\clearpage
\section{Algorithms}

\subsection{Algorithms in the independent and non independent case with details}

Algorithms~\ref{alg:detailled} and~\ref{alg:detailled_nind} includes practical details as warming and soft-start described in Section~\ref{sec:practical} in respectively the independent case and the non-independent case.
The authors uses the following parameters:
\begin{itemize}
	\item $\gamma_0=10^{-4}$,
	\item $K_{\text{pre-heating}}=1000$,
	\item $C_{\text{heating}}=\frac1{1000}$,
	\item $\alpha=\frac23$,
\end{itemize}

\begin{algorithm}
	\caption{Fisher-SGD in the independent case with details}
	\label{alg:detailled}
	\begin{algorithmic} 
		\STATE{\bfseries Input:} $z_0,\para_0$, $y_1,\dots,y_N$
		\STATE{\bfseries Algorithm parameters:} $\gamma_0, K_{\text{pre-heating}}, C_{\text{heating}}, \alpha$
		\STATE Initialize a 3-order mean filter with constant $C_{\text{heating}}$ to compute mean grad.
		\STATE{Set heating as not finished}
		\FOR{\iter=1,\dots,K}
		\STATE\COMMENT{\textit{The following loop is operated with vector calculus without explicit loop.}}
		\FOR{ i=1,\dots,N}
		\STATE{$z_i^{\iter} \sim \kernel_{\para_{\iter-1}} \p{\cdot \cond z_i^{\iter-1}, \obs_i}$ where $\kernel_{\para} \p{\cdot \cond \cdot, \obs_i}$ is a transition kernel of a MCMC having $\post_\para\p{ \cdot \cond y_i }$ as stationary distribution for all $\para$}
		\ENDFOR
		\IF{$k<K_{\text{pre-heating}}$}
		\STATE{Set $\gamma_k=\exp\p{\p{1-k/K_{\text{pre-heating}}}\log \gamma_0}$}
		\ELSIF{heating not finished}
		\STATE{Set $\gamma_k=1$}
		\ELSE
		\STATE{Set $\gamma_k=\p{k-K_{\text{end-heating}}}^{-\alpha}$}
		\ENDIF
		\STATE\COMMENT{\textit{The following loop is computed without explicit loop, all the gradients are computed in one step as the jacobian of the vector of criterion, and $\Delta$ with vector calculus.}}
		\FOR{ i=1,\dots,N} 
		\STATE{ Compute $J_i^{\iter}=\nabla_\para\log\comp\p{y_i,z_i^\iter;\para_{\iter-1}}$}
		\STATE{ Compute $\Delta_i^{\iter}=\p{1-\pas_\iter}\Delta_i^{\iter-1}+\pas_\iter J_i^\iter$}
		\ENDFOR
		\STATE{Compute $v_\iter=\frac{1}{\nobs}\sum_{i=1}^\nobs J_i^\iter$}
		\STATE\COMMENT{\textit{The following loop is computed with vector calculus}}
		\STATE{Compute  $I_\iter=\dfrac{1}{N}\sum_{i=1}^N\Delta_i^{\iter}\p{\Delta_i^{\iter}}^T$}
		\IF{$k<K_{\text{pre-heating}}$}
		\STATE{Set $P_k=\p{1-\gamma_k}\max\p{1,\operatorname{tr}\p{I_\iter }}Id+\gamma_k I_\iter$}
		\ELSE
		\STATE{Set $P_k=I_\iter$}
		\ENDIF
		\IF{$k>K_{\text{pre-heating}}$ and heating not finished}
		\STATE{Update 3-order mean filter with $v_k$}
		\IF{norm-2 of 3-order mean of gradient is increasing}
		\STATE{Set heating finished}
		\STATE{Set $K_{\text{end-heating}}=k$}
		\ENDIF
		\ENDIF
		\STATE{Set $\para_\iter=\para_{\iter-1}+\pas_\iter
			P_k^{-1}v_\iter$}
		\ENDFOR
		\STATE{\bfseries Output:} Estimated parameter: $\para_\iter$, estimated Fisher information matrix of the whole sample:  $NI_k$, sample of latent in the posterior distribution: $\p{z^{K-r+1}\cdots z^K}$.
	\end{algorithmic}
\end{algorithm}

\begin{algorithm}
	\caption{Fisher-SGD in the non independent case with details}
	\label{alg:detailled_nind}
	\begin{algorithmic} 
		\STATE{\bfseries Input:} $z_0,\para_0$, $y$
		\STATE{\bfseries Algorithm parameters:} $\gamma_0, K_{\text{pre-heating}}, C_{\text{heating}}, \alpha$
		\STATE Initialize a 3-order mean filter with constant $C_{\text{heating}}$ to compute mean grad.
		\STATE{Set heating as not finished}
		\FOR{\iter=1,\dots,K}
		\STATE{$z^{\iter} \sim \kernel_{\para_{\iter-1}} \p{\cdot \cond z^{\iter-1}, \obs}$ where $\kernel_{\para} \p{\cdot \cond \cdot, \obs}$ is a transition kernel of a MCMC having $\post_\para  \p{ \cdot \cond y }$ as stationary distribution for all $\para$}
		\IF{$k<K_{\text{pre-heating}}$}
		\STATE{Set $\gamma_k=\exp\p{\p{1-k/K_{\text{pre-heating}}}\log \gamma_0}$}
		\ELSIF{heating not finished}
		\STATE{Set $\gamma_k=1$}
		\ELSE
		\STATE{Set $\gamma_k=\p{k-K_{\text{end-heating}}}^{-\alpha}$}
		\ENDIF
		\STATE{Compute $v_\iter=\nabla_\para\log\comp\p{y,z^\iter;\para_{\iter-1}}$}
		\STATE{Compute  $I_\iter$ with custom method.}
		\IF{$k<K_{\text{pre-heating}}$}
		\STATE{Set $P_k=\p{1-\gamma_k}\max\p{1,\operatorname{tr}\p{I_\iter}}Id+\gamma_k I_\iter$}
		\ELSE
		\STATE{Set $P_k=I_\iter$}
		\ENDIF
		\IF{$k>K_{\text{pre-heating}}$ and heating not finished}
		\STATE{Update 3-order mean filter with $v_k$}
		\IF{norm-2 of 3-order mean of gradient is increasing}
		\STATE{Set heating finished}
		\STATE{Set $K_{\text{end-heating}}=k$}
		\ENDIF
		\ENDIF
		\STATE{Set $\para_\iter=\para_{\iter-1}+\pas_\iter
			P_k^{-1}v_\iter$}
		\IF{$k>K-r$}
		\STATE{Compute $\nabla_\para^2\log\comp\p{y,z^\iter;\para_\iter}$}
		\ENDIF
		\ENDFOR
		\STATE{\bfseries Output:} Estimated parameter: $\para_\iter$, estimated Fisher information matrix of the whole sample: $\frac1r\sum_{k=K-r+1}^K\nabla_\para^2\log\comp\p{y,z^\iter;\para_{k}}$, sample of latent in the posterior distribution: $\p{z^{K-r+1}\cdots z^K}$.
	\end{algorithmic}
\end{algorithm}

\begin{algorithm}
	\caption{Fisher-SGD for the Stochastic Block Model with details}
	\label{alg:SBM}
	\begin{algorithmic} 
		\STATE{\bfseries Input:} $z_0,\para_0$, $y$
		\STATE{\bfseries Algorithm parameters:} $\gamma_0, K_{\text{pre-heating}}, C_{\text{heating}}, \alpha$
		\STATE Initialize a 3-order mean filter with constant $C_{\text{heating}}$ to compute mean grad.
		\STATE{Set heating as not finished}
		\FOR{\iter=1,\dots,K}
		\STATE\COMMENT{\textit{Here, a Gibbs sampler is used}}
		\FOR{i in random permutation of 1,\dots,N}
		\STATE{$z_i^{\iter} \sim \comp\p{\cdot \cond \p{z_j^{\iter-1}}_{j\neq i}, y;
				\para_{\iter-1}}$ }
		\ENDFOR
		\IF{$k<K_{\text{pre-heating}}$}
		\STATE{Set $\gamma_k=\exp\p{\p{1-k/K_{\text{pre-heating}}}\log \gamma_0}$}
		\ELSIF{heating not finished}
		\STATE{Set $\gamma_k=1$}
		\ELSE
		\STATE{Set $\gamma_k=\p{k-K_{\text{end-heating}}}^{-\alpha}$}
		\ENDIF
		\STATE\COMMENT{\textit{The following loops is computed without explicit loops, all the gradients are computed in one step as the jacobian of the vector of criterion, and $\Delta^{\text{obs}}$ with vector calculus.}}
		\FOR{ i in 1,\dots,N} 
		\FOR{j in 1,\dots,N}
		\STATE{Compute $J_{ij}^{\iter,\text{obs}}=\nabla_\para\log\comp\p{y_{ij}|z_i^\iter,z_j^\iter;\para_{\iter-1}}$}
		\STATE{Compute $\Delta^{\iter,\text{obs}}_{ij}=\p{1-\pas_\iter}\Delta^{\iter-1,\text{obs}}_{ij}+\pas_\iter J_{ij}^{\iter,\text{obs}}$}
		\ENDFOR
		\ENDFOR
		\STATE\COMMENT{\textit{The following loop is computed without explicit loop, all the gradients are computed in one step as the jacobian of the vector of criterion, and $\Delta^{\text{lat}}$ with vector calculus.}}
		\FOR{ i in 1,\dots,N} 
		\STATE{Compute $J_{i}^{\iter,\text{lat}}=\nabla_\para\log\comp\p{z_{i}^\iter;\para_{\iter-1}}$}
		\STATE{Compute
			$\Delta^{\iter,\text{lat}}_{i}=\p{1-\pas_\iter}\Delta^{\iter-1,\text{lat}}_{i}+\pas_\iter
			J_{i}^{\iter,\text{lat}}$}
		\ENDFOR
		\STATE{Compute $v_\iter=\sum_{ij}J^{\iter,\text{obs}}_{ij}+\sum_iJ^{\iter,\text{lat}}_i$}
		\STATE\COMMENT{\textit{The following computation is computed efficiently with Einstein summation}}
		\STATE{Compute  $I_\iter=\sum_{ij}\Delta^{\iter,\text{obs}}_{ij}\p{\Delta^{\iter,\text{obs}}_{ij}}^T+\sum_{i}\Delta^{\iter,\text{lat}}_{i}\p{\Delta^{\iter,\text{lat}}_{i}}^T$}
		\IF{$k<K_{\text{pre-heating}}$}
		\STATE{Set $P_k=\p{1-\gamma_k}\max\p{1,\operatorname{tr}\p{I_\iter}}Id+\gamma_k I_\iter$}
		\ELSE
		\STATE{Set $P_k=I_\iter$}
		\ENDIF
		\IF{$k>K_{\text{pre-heating}}$ and heating not finished}
		\STATE{Update 3-order mean filter with $v_k$}
		\IF{norm-2 of 3-order mean of gradient is increasing}
		\STATE{Set heating finished}
		\STATE{Set $K_{\text{end-heating}}=k$}
		\ENDIF
		\ENDIF
		\STATE{Set $\para_\iter=\para_{\iter-1}+\pas_\iter
			P_k^{-1}v_\iter$}
		\ENDFOR
		\STATE{\bfseries Output:} Estimated parameter: $\para_\iter$, estimated Fisher information matrix of the whole sample: $\frac1r\sum_{k=K-r+1}^K\nabla_\para^2\log\comp\p{y,z^\iter;\para_{k}}$, sample of latent in the posterior distribution: $\p{z^{K-r+1}\cdots z^K}$.
	\end{algorithmic}
\end{algorithm}

\FloatBarrier

\clearpage
\section{Proofs of the theorems}\label{sec:app:proofs}

Let us first recall the expression of the FIM estimate proposed by  \citet{delattre2019}: 
$$\estFIM{\para} = \dfrac{1}{N}\sum_{i=1}^N \expectation\p{ \nabla_\para \log \comp\p{\obs_i,Z_i;\para} \cond \obs_i ;\para} 
\expectation\p{ \nabla_\para \log \comp\p{\obs_i,Z_i;\para} \cond \obs_i ;\para}^T.    $$

\subsection{First setting: realizations of latent are generated from the posterior distribution}

We adapt the proof of the convergence of the standard stochastic gradient algorithms (see, e.g. \citet{Bottou} or \citet{Ghadimi}) to our algorithm. The specificity in our proof comes from the control of the preconditioner term.

\begin{proof}[Proof of Theorem~\ref{th:prop_post}]
	By Taylor-Lagrange inequality and under Assumption~\ref{ass:nonconvex}.\ref{ass:L}, we get for all $k$
	\begin{align*}
		F(\para_{k+1})\leq& F(\para_k)+\langle\nabla_\para F(\para_k),\para_{k+1}-\para_k\rangle+\dfrac{L}{2}\|\para_{k+1}-\para_k\|^2\\
		\leq& F(\para_k)+\pas_k\langle\nabla_\para F(\para_k),\estFIM{\para_\iter}^{-1}\nabla_\para\log \comp(\obs,\lat^{k+1};\para_k)\rangle\\
		&+\pas_k^2\dfrac{L}{2}\|\estFIM{\para_\iter}^{-1}\nabla_\para\log\comp(\obs,\lat^{k+1};\para_k)\|^2
	\end{align*}
	
	We introduce some notations to  simplify the writing of the proof. Let $H_\para(\lat)=-\nabla_\para \log f(\obs,\lat;\para)$.
	We introduce also the following notation $\expectation_{k}$ for the  expectation taking with respect to the posterior distribution $\post_{\para_k}$.

	Taking the conditional expectation in the previous inequality and noting that $\expectation_{k}(H_{\para_k}(\lat^{k+1})  )= \nabla_\para F(\para_k)$, we get:
	\begin{align*}
		\expectation_{k}[F(\para_{\iter+1})]\leq &F(\para_\iter)-\pas_\iter\langle\nabla_\para F(\para_\iter),\estFIM{\para_\iter}^{-1}\expectation_{k}
		(H_{\para_k}(\lat^{k+1}))
		\rangle\\
		&+\pas_\iter^2\dfrac{L}{2}\expectation_{k}[\|\estFIM{\para_\iter}^{-1}H_{\para_k}(\lat^{k+1}) \|^2]\\
		\leq& F(\para_\iter)-\pas_\iter\|\hat{I}(\para_\iter)^{-1/2}\nabla_\para F(\para_\iter)\|^2\\
		&+\pas_\iter^2\dfrac{L}{2}\expectation_{k}[\|\estFIM{\para_\iter}^{-1}H_{\para_k}(\lat^{k+1}) \|^2]
	\end{align*}
	Under Assumption~\ref{ass:nonconvex}.\ref{ass:C} and~\ref{ass:eigen}, taking the total expectation and the sum for $l$ between 0 and $\iter$, we obtain 
	\[
	\mu_{m}\expectation\left(\sum_{l=0}^\iter\pas_l\|\nabla_\para F(\para_l)\|^2\right)\leq F(\para_0)-\expectation(F(\para_{\iter+1}))+\mu_{M}^2\dfrac{CL}{2}\sum_{l=0}^\iter\pas_l^2.
	\]
	We finally obtain the result by noticing that for all $0 \leq l \leq k$ $\|\nabla F(\para_l)\|^2\geq\min_{0\leq l'\leq \iter}\|\nabla F(\para_{l'})\|^2$ for all $l$ and $\expectation[F(\para_{k+1})]\geq \min F$. 
\end{proof}

\subsection{Second setting: realizations of latent are generated from a  transition kernel from an ergodic  Markov chain having the posterior distribution as stationary distribution}

\begin{proof}[Proof of Theorem~\ref{th:prop_kernel}]
	
	We consider now the setting where at iteration $k$ the realization $\lat^k$ is sampled from a transition kernel of a Markov chain having the posterior distribution as stationary distribution. To state the convergence proof in this setting let us 
	introduce for all integer $k$ the $\sigma-$algebra $\mathcal{F}_k=\sigma(\para_0, \lat_l, 0 \leq l \leq k  )$. 
	
	In this case the expectation $\expectation(H_{\para_k}(\lat^{k+1})  | \mathcal{F}_k )$
	is not equal to $ \nabla_\para F(\para_k)$, leading to the presence of a supplementary  term due to the use of a MCMC in the simulation task. The proof begins the same way as in the previous case.
	
	By Taylor-Lagrange inequality and under Assumption~\ref{ass:nonconvex}.\ref{ass:L}, we get for all $k$
	\begin{align*}
		F(\para_{k+1})\leq& F(\para_k)+\langle\nabla_\para F(\para_k),\para_{k+1}-\para_k\rangle+\dfrac{L}{2}\|\para_{k+1}-\para_k\|^2\\
		\leq& F(\para_k)+\pas_k\langle\nabla_\para F(\para_k),\estFIM{\para_\iter}^{-1}\nabla_\para\log \comp(\obs,\lat^{k+1};\para_k)\rangle\\
		&+\pas_k^2\dfrac{L}{2}\|\estFIM{\para_\iter}^{-1}\nabla_\para\log\comp(\obs,\lat^{k+1};\para_k)\|^2
	\end{align*}
	
	Taking first the expectation  conditionally to the $\sigma-$algebra $\mathcal{F}_k$ and introducing $ \nabla_\para F(\para_k)$, we get:
	
	\begin{align*}
		\expectation(F(\para_{k+1})| \mathcal{F}_k) \leq& F(\para_k)-\pas_k \langle\nabla_\para F(\para_k),
		\estFIM{\para_\iter}^{-1}\nabla_\para F(\para_k) \rangle \\
		&+\pas_k \langle \nabla_\para F(\para_k),
		\estFIM{\para_\iter}^{-1}(\nabla_\para F(\para_k)-\expectation(H_{\para_k}(\lat^{k+1})  | \mathcal{F}_k ) )\rangle \\
		&+\pas_k^2\dfrac{L}{2}\expectation (\|\estFIM{\para_\iter}^{-1}H_{\para_k}(\lat^{k+1}) \|^2 | \mathcal{F}_k)
	\end{align*}
	
	The difficulty here is to control the additional term $B_k= \langle \nabla_\para F(\para_k),
	\estFIM{\para_\iter}^{-1}(\nabla_\para F(\para_k)-\expectation(H_{\para_k}(\lat^{k+1})  | \mathcal{F}_k ) )\rangle $ in the second line raised up by the MCMC procedure used for the simulation of $\lat^k$.
	
	Let us introduce the notation $\eta_k= H_{\para_k}(\lat^{k+1}) - \nabla_\para F(\para_k) $.
	
	Taking full expectation of the previous inequality, we get
	\begin{align*}
		\expectation(F(\para_{k+1})) \leq& \expectation(F(\para_k))-\pas_k  \expectation(\langle \nabla_\para F(\para_k),
		\estFIM{\para_\iter}^{-1}\nabla_\para F(\para_k) \rangle) \\
		&-\pas_k \expectation( \langle \nabla_\para F(\para_k),
		\estFIM{\para_\iter}^{-1} \expectation (\eta_k  | \mathcal{F}_k ) \rangle) \\
		&+\pas_k^2\dfrac{L}{2}\expectation (\|\estFIM{\para_\iter}^{-1}H_{\para_k}(\lat^{k+1}) \|^2 )
	\end{align*}
	Therefore reordering the terms we get
	\begin{align*}
		\pas_k  \expectation(\langle \nabla_\para F(\para_k),
		\estFIM{\para_\iter}^{-1}\nabla_\para F(\para_k) \rangle) \leq&
		\expectation(F(\para_k)) -\expectation(F(\para_{k+1})) \\
		&-\pas_k \expectation( \langle \nabla_\para F(\para_k),
		\estFIM{\para_\iter}^{-1} \expectation( \eta_k  | \mathcal{F}_k ) \rangle )\\
		&+\pas_k^2\dfrac{L}{2}\expectation (\|\estFIM{\para_\iter}^{-1}H_{\para_k}(\lat^{k+1}) \|^2 )
	\end{align*}

	Under Assumption~\ref{ass:nonconvex}.\ref{ass:C} and~\ref{ass:nonconvex}.\ref{ass:eigen},  suming for $l$ between 0 and $\iter$, we obtain 
	\begin{align*}
		\mu_{m}\expectation\left(\sum_{l=0}^\iter\pas_l\|\nabla_\para F(\para_l)\|^2\right)&\leq F(\para_0)-\expectation(F(\para_{\iter+1}))+\mu_{M}^2\dfrac{CL}{2}\sum_{l=0}^\iter\pas_l^2\\
		&- \expectation (\sum_{l=0}^\iter\pas_l \langle \nabla_\para F(\para_l),\estFIM{\para_l}^{-1} \expectation( \eta_l  | \mathcal{F}_l )  \rangle   )
	\end{align*}

	To control the last term, we apply the result of 
	Lemma~\ref{lem:serie} below which proves that $\sum \pas_k \langle \nabla_\para F(\para_k),
	\estFIM{\para_\iter}^{-1}  \eta_k   \rangle$ converges almost surely. 
	Therefore we get:
	\begin{align*}
		\mu_{m}\expectation\left(\sum_{l=0}^\iter\pas_l\|\nabla_\para F(\para_l)\|^2\right)&\leq F(\para_0)-\expectation(F(\para_{\iter+1}))+\mu_{M}^2\dfrac{CL}{2}\sum_{l=0}^\iter\pas_l^2 +C
	\end{align*}
	
	Finally we get the result by dividing the inequality  by $\sum_{l=0}^\iter\pas_l$.
	
\end{proof}

Before  stating  Lemma~\ref{lem:serie}, we first establish several preliminary results needed for the proof. In particular we  first introduce the Poisson equation  as done for example in \citet{allassonniere2015,atchade2017perturbed} and establish several technical lemmas derived below. Recall $H_\para(\lat)=-\nabla_\para \log f(\obs,\lat;\para)$ and $\eta_k= H_{\para_k}(\lat^{k+1}) - \nabla_\para F(\para_k) $.

\begin{lemma}
	\label{lem:poisson}
	Assume~\ref{ass:MCMC}. Then there exists a measurable function $(\para,z)\rightarrow \hat{H}_\para(z)$ such that $\sup_\para |\hat{H}_\para |_W<\infty$ and for any $(\para,z)\in \paraspace \times \latspace$, 
	\begin{equation}
		\label{eq:poisson}
		\hat{H}_\para(z) -\kernel_\para \hat{H}_\para(z) = H_\para(z) - \int H_\para(z) \post_\para(z) dz 
	\end{equation}
	Moreover there exists a constant $C$ such that for any $(\para,\para')\in \paraspace^2$,
	$$
	\|\kernel_\para \hat{H}_\para(z)-\kernel_{\para'} \hat{H}_{\para'}(z)\|_W \leq C \|\para-\para'\|
	$$
\end{lemma}

\begin{proof}[Proof of Lemma~\ref{lem:poisson}]
	The proof is established in details in Lemma 4.2 of \cite{fort2011}.
\end{proof}

\begin{lemma}
	\label{lem:controlw}
	Under Assumptions~\ref{ass:MCMC}.\ref{ass:kernel},\ref{ass:MCMC}.\ref{ass:drift}, we have  
	$\sup_k \expectation (W^p(Z^k)) <\infty$.
\end{lemma}

\begin{proof}[Proof of Lemma~\ref{lem:controlw}]
	Since $Z_{k+1}$ is generated from the transition kernel $\kernel_{\para_k}(\cdot| Z^k, y)$, we get :
	$$\expectation (W^p(Z^{k+1}))=\expectation (\expectation (W^p(Z^{k+1})| \mathcal{F}_k  ))=  \expectation ( \kernel_{\para_k} W^p(Z^k) )
	$$
	Applying the drift inequality of Assumption~\ref{ass:MCMC}.\ref{ass:drift}, we get 
	$$ \expectation (W^p(Z^{k+1})) \leq  \lambda \expectation (W^p(Z^{k})) +b $$ 
	The result is then obtained by induction.
\end{proof}

\begin{lemma}
	\label{lem:serie}
	Assume Assumptions~\ref{ass:MCMC}.\ref{ass:kernel}, \ref{ass:MCMC}.\ref{ass:drift}, \ref{ass:MCMC}.\ref{ass:ergo}, \ref{ass:stepsizeadd} and $\paraspace$ is bounded.  Then 
	$\sum \pas_k \langle \nabla_\para F(\para_k),
	\estFIM{\para_\iter}^{-1} \eta_k  \rangle$ converges almost surely.
\end{lemma}

\begin{proof}[Proof of Lemma~\ref{lem:serie}]
	Applying Lemma~\ref{lem:poisson}, we get that there exist a function $\hat{H}_{\para_k}$ satisfying equation (\ref{eq:poisson}). Therefore we get $\eta_k=\hat{H}_{\para_k}(Z^{k+1})-   \kernel_{\para_k}\hat{H}_{\para_k}(Z^{k+1}) $.
	Let us denote by $M_k=\hat{H}_{\para_k}(Z^{k+1})-\kernel_{\para_k}\hat{H}_{\para_k}(Z^{k})$, $R_k=\kernel_{\para_k}\hat{H}_{\para_k}(Z^{k})- \kernel_{\para_{k+1}}\hat{H}_{\para_{k+1}}(Z^{k+1})$, 
	and $R_k'=\kernel_{\para_{k+1}}\hat{H}_{\para_{k+1}}(Z^{k+1})  -\kernel_{\para_k}\hat{H}_{\para_k}(Z^{k+1}) $ such that $\eta_k=M_k+R_k+R_k'$. We will prove successively that  the three sums $\expectation (\sum \pas_k \| \langle \nabla_\para F(\para_k),
	\estFIM{\para_\iter}^{-1} M_k \rangle \|)$ , $\expectation (\sum \pas_k \| \langle \nabla_\para F(\para_k),
	\estFIM{\para_\iter}^{-1} R_k \rangle \|)$ , $\expectation (\sum \pas_k \| \langle \nabla_\para F(\para_k),
	\estFIM{\para_\iter}^{-1} R_k' \rangle \|)$ are finite with probability one.
	
	Let us first note that the term $\pas_k \langle \nabla_\para F(\para_k),
	\estFIM{\para_\iter}^{-1} M_k \rangle$  is a martingale increment with respect to the filtration $(\mathcal{F}_k)$. By Lemma~\ref{lem:poisson} and under Assumption~\ref{ass:MCMC}.\ref{ass:drift}, we get that there exists $C$ such that with probability one for all $k$
	$\|M_k \| \leq C(W(Z^{k+1})+W(Z^k))$.
	Applying classical result on martingales \cite{hall2014}, we get that $\sum \pas_k \langle \nabla_\para F(\para_k),
	\estFIM{\para_\iter}^{-1} M_k \rangle$ converges almost surely if $\sum \pas_k^2 \| \estFIM{\para_\iter}^{-1} \nabla_\para F(\para_k)\|^2 \|M_k\|^2<\infty$ almost surely. By Assumption~\ref{ass:model}, $\| \estFIM{\para_\iter}^{-1}\nabla_\para F(\para_k)\|^2\leq C\mu_M^2 $ and  $\|M_k \|^2 \leq 2C(W^2(Z^{k+1})+W^2(Z^k))$ leading to 
	$\expectation (\sum \pas_k^2 \| \estFIM{\para_\iter}^{-1} \nabla_\para F(\para_k)\|^2 \|M_k\|^2)<\infty$ which gives the control of the second sum.

	Concerning the second term, applying the Abel transformation leads to 
	\begin{align*}
		&\sum_{k=0}^K \pas_k \langle \nabla_\para F(\para_k),
		\estFIM{\para_\iter}^{-1} R_k \rangle=\sum_{k=1}^K  \langle \pas_{k}
		\estFIM{\para_\iter}^{-1}\nabla_\para F(\para_k) - \pas_{k-1}
		\estFIM{\para_{\iter-1}}^{-1}\nabla_\para F(\para_{k-1}), \kernel_{\para_k} \hat{H}_{\para_k} (Z^k)  \rangle \\&+ 
		\pas_0 \langle \estFIM{\para_0}^{-1} \nabla_\para F(\para_0), \kernel_{\para_0} \hat{H}_{\para_0} (Z^0) 
		\rangle - 
		\pas_K \langle \estFIM{\para_K}^{-1} \nabla_\para F(\para_K), \kernel_{\para_{K+1}} \hat{H}_{\para_{K+1}} (Z^{K+1}) 
		\rangle
	\end{align*}
	
	Let us denote by $\xi_k=\pas_{k}
	\estFIM{\para_\iter}^{-1}\nabla_\para F(\para_k) - \pas_{k-1}
	\estFIM{\para_{\iter-1}}^{-1}\nabla_\para F(\para_{k-1})$. 
	Therefore we get
	\begin{align*}
		\|\xi_k \| &\leq |\pas_{k}-\pas_{k-1}| \|\estFIM{\para_\iter}^{-1}\nabla_\para F(\para_k)\|
		+ \pas_{k-1}\|\estFIM{\para_\iter}^{-1}\nabla_\para F(\para_k) -\estFIM{\para_\iter}^{-1}\nabla_\para F(\para_{k-1})  \| \\
		&+\pas_{k-1}\|\estFIM{\para_\iter}^{-1}\nabla_\para F(\para_{k-1}) -\estFIM{\para_{\iter-1}}^{-1}\nabla_\para F(\para_{k-1})  \|\\
	\end{align*}
	Under Assumption~\ref{ass:model}, there exists $M>0$ such that for all $k$ $\|\estFIM{\para_\iter}^{-1}- \estFIM{\para_{\iter-1}}^{-1}\| \leq M \|
	\para_\iter-\para_{k-1}\| $ therefore we get
	\begin{align*}
		\|\xi_k \| &\leq |\pas_{k}-\pas_{k-1}| \|\estFIM{\para_\iter}^{-1}\nabla_\para F(\para_k)\|
		+ \pas_{k-1} \mu_M \|\nabla_\para F(\para_k) -\nabla_\para F(\para_{k-1}) \| \\
		&+\pas_{k-1} C M \|
		\para_\iter-\para_{k-1}\|\\
		&\leq |\pas_{k}-\pas_{k-1}| C \mu_M + \pas_{k-1} \mu_M L \|
		\para_\iter-\para_{k-1}\| + \pas_{k-1} C M \|
		\para_\iter-\para_{k-1}\| 
	\end{align*}
	Taking the expectation, we get 
	\begin{align*}
		\expectation(\|\xi_k \|) \leq |\pas_{k}-\pas_{k-1}| C \mu_M + \pas^2_{k-1}C (\mu_M L +CM) 
	\end{align*}
	By Lemma~\ref{lem:poisson} and Assumptions~\ref{ass:MCMC}.\ref{ass:drift} and~\ref{ass:MCMC}.\ref{ass:ergo} there exists $C$ such that with probability one for all $k$ 
	$\| \kernel_{\para_k} \hat{H}_{\para_k} (Z^k) \| \leq CW(Z^k)$. Moreover
	by Lemma~\ref{lem:controlw}, we get $\sup_k \expectation (W^2(Z^k))<\infty$ thus  we get   $\sum \expectation (\|  \pas_k \langle \nabla_\para F(\para_k),
	\estFIM{\para_\iter}^{-1} R_k \rangle    \| )<\infty$ which control the second sum.

	Finally let us consider the third sum. By Lemma~\ref{lem:poisson} there exists $C$ such that with probability one for all $k$ 
	\begin{align*}
		\| \langle \nabla_\para F(\para_k), \estFIM{\para_\iter}^{-1} R_k' \rangle \| &\leq C \mu_M \|
		\para_\iter-\para_{k+1}\|   W(Z^{k+1})  \\  
		&\leq C \mu_M^2  \pas_{k+1}  W^2(Z^{k+1}) 
	\end{align*}
	By Lemma~\ref{lem:controlw} we get $\sum \pas_k^2 \expectation (W^2(Z^{k+1}))<\infty$ which implies that 
	$\expectation (\sum \pas_k \| \langle \nabla_\para F(\para_k),
	\estFIM{\para_\iter}^{-1} R_k' \rangle \|)$ is finite almost surely.
	
	This concludes the proof of Lemma~\ref{lem:serie}.
\end{proof}

\FloatBarrier

\clearpage
\section{Reparametrization of models and confidence intervals}\label{sec:app:reparam}

\subsection{Reparametrization}

Models in numerical experiments use a constrained parameter space.

In particular, we have:

\begin{description}
	\item[Logistic growth mixed-effects model:] parameters are $\beta \in \mathbb{R}^*_+ \times \mathbb R$, $\alpha \in \mathbb R^*_+$, $\Gamma \in \mathbb S^2_{++}$, where $\mathbb S^{++}_p$ is the set of symmetric, positive definite matrices of size $p \times p$ and $\sigma \in \mathbb R^*_+$. Then, the original parameter space is $\mathbb R_+^*\times\mathbb R\times\mathbb R_+^*\times \mathbb{S}^{++}_2\times\mathbb R_+^*$ where $\mathsf{S}^{++}_2$ is the set of $2\times2$ positive definite matrix.
	\item[Stochastic Block Model:] parameters are $\alpha\in\mathring{\mathcal S}_{K-1}$ and $p\in\ouv01^{K\times K}$ where $\mathring{\mathcal S}_{K-1}\subset \mathbb R^K$ is the $K-1$ dimentionnal open unit simplex.
\end{description}

Handling directly this parameter space in our algorithm could lead to constraint violations. To handle constraints we propose to use a bijective differentiable mapping from the constrained parameter space to $\mathbb R^d$.

Reparametrization are build using the Parametrization Cookbook \cite{paramcookbook}, and practically using the Python module \texttt{parametrization\_cookbook}.

We introduce the original parameter as a function of $\para\in\paraspace=\mathbb R^d$.

\begin{description}
	\item[Logistic growth mixed-effects model:] $\paraspace=\mathbb R^d$ with $d=7$, and we use $\alpha_\para$, $\beta_\para$, $\Gamma_\para$ and $\sigma_\para$.
	\item[Stochastic Block Model:] $\paraspace=\mathbb R^d$ with $d={K^2+K-1}$, and we use $\alpha_\para$ and $p_\para$.
\end{description}

\subsection{Confidence ellipsoid}

We obtain $\widehat\theta$ an estimate of $\theta_0$ and $\estFIMwhole{\widehat\theta}$ the estimation of the Fisher Information Matrix of the whole sample. Therefore we have asymptotically:
\[
\estFIMwhole{\widehat\theta}^{1/2}\p{\widehat\theta-\theta_0}\xrightarrow[n\to+\infty]{(d)}\mathcal N\p{0,I_d}.
\]

So we compute the asymptotic confidence ellipsoid on $\theta_0$ at confidence level $1-a$:
\[
\braces{\theta\in\mathbb R^d\;\colon\; \p{\theta-\widehat\theta}^T\brackets{\estFIMwhole{\widehat\theta}}\p{\theta-\widehat\theta} \le \chi^2_{d; 1-a}.
}
\]

\subsection{Confidence interval on original parameter}

With $\eta$ an original parameter of the model before reparametrization (see section above for original parameter in logistic growth mixed-effects model or in stockastic block model). After reparametrization this original parameter is a function of $\theta$, noted $\eta_\theta$.

Applying the reparametrization, from $\widehat\theta$, we obtain $\eta_{\widehat\theta}$ an estimate of $\eta_0$.

Applying delta-method \cite{vandervaart2000}, we obtain the asymptotic distribution of $\eta_{\widehat\theta}$:

\[
\frac{\eta_{\widehat\theta}-\eta_0}{\sqrt{g_{\eta_{\widehat\theta}}^T\brackets{\estFIMwhole{\widehat\theta}}^{-1}g_{\eta_{\widehat\theta}}}}
\xrightarrow[n\to+\infty]{(d)}
\mathcal N(0,1),
\]
with
\[
g_{\eta_{\widehat\theta}} = {\left.\frac{\operatorname{d}\eta_\theta}{\operatorname{d}\theta}\right|}_{\theta=\widehat\theta},
\]
and $g_{\eta_{\widehat\theta}}$ is computed with automatic differentiation.

Then we obtain the asymptotic confidence interval on $\eta_0$ at confidence level $(1-a)$:

\[
\brackets{
	\eta_{\widehat\theta}\pm u_{1-a/2}\sqrt{g_{\eta_{\widehat\theta}}^T\brackets{\estFIMwhole{\widehat\theta}}^{-1}g_{\eta_{\widehat\theta}}}
}.
\]

\FloatBarrier

\clearpage
\section{Supplementary numerical results}

\subsection{Logistic growth mixed-effects model}\label{sec:app_logistic_model}

\begin{figure}[ht]
	\centering
	\includegraphics[scale=0.65]{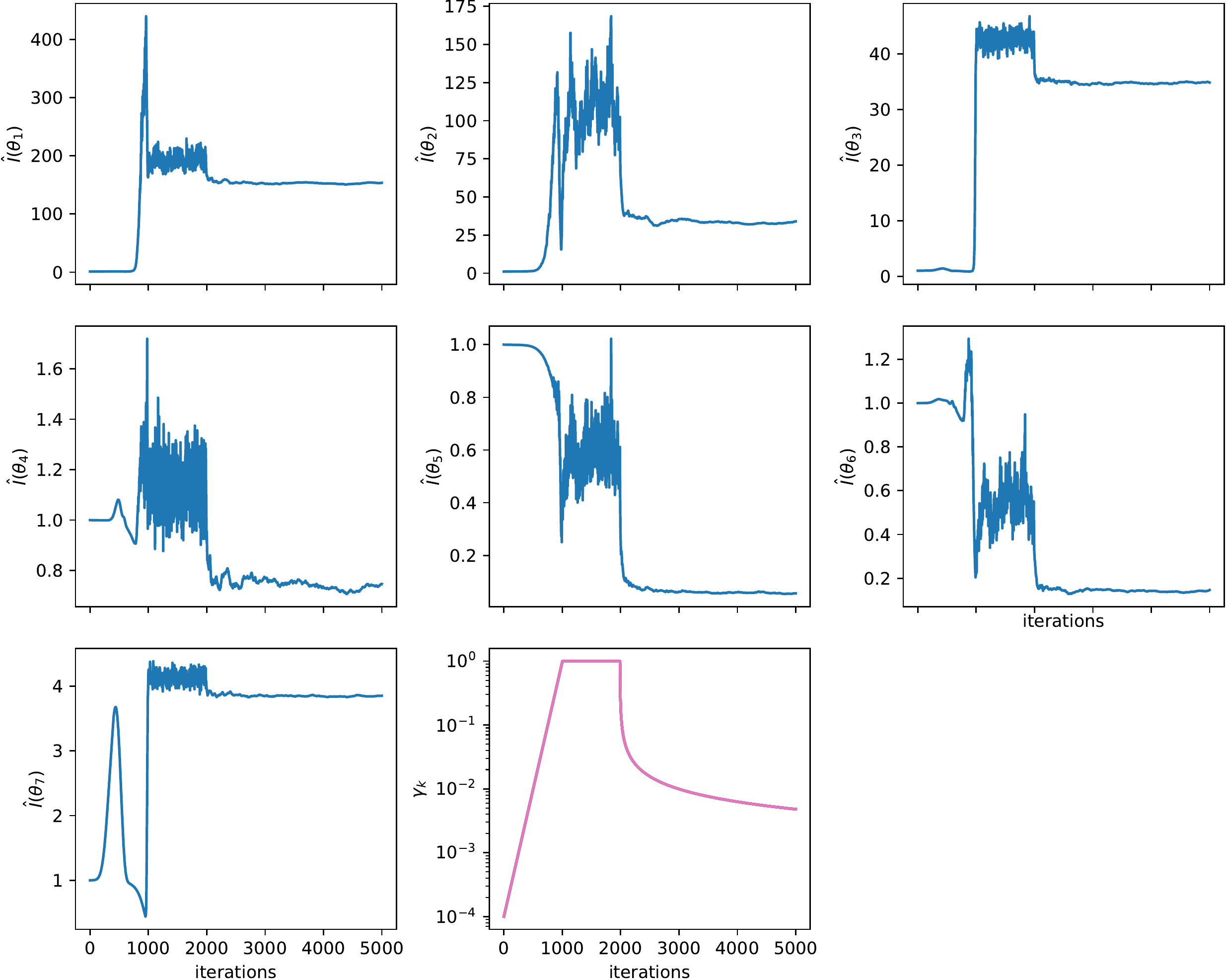}
	\caption{Evolution of the diagonal elements of the FIM estimate across the iterations (in blue), along with the evolution of the learning step (in red).}
	\label{fig:evol_fim}
\end{figure}

\subsection{Real data analysis}

We applied our algorithm to a real dataset from a study on coucal growth rates \citep{Goy16}. In this study, body weights of $\nobs = 259$ birds were collected from their hatching date until they left the nest (see Figure \ref{fig:coucaldata}). The number of measurements per bird ranges from 1 to 9.

\begin{figure}
	\centering
	\includegraphics[scale=0.5]{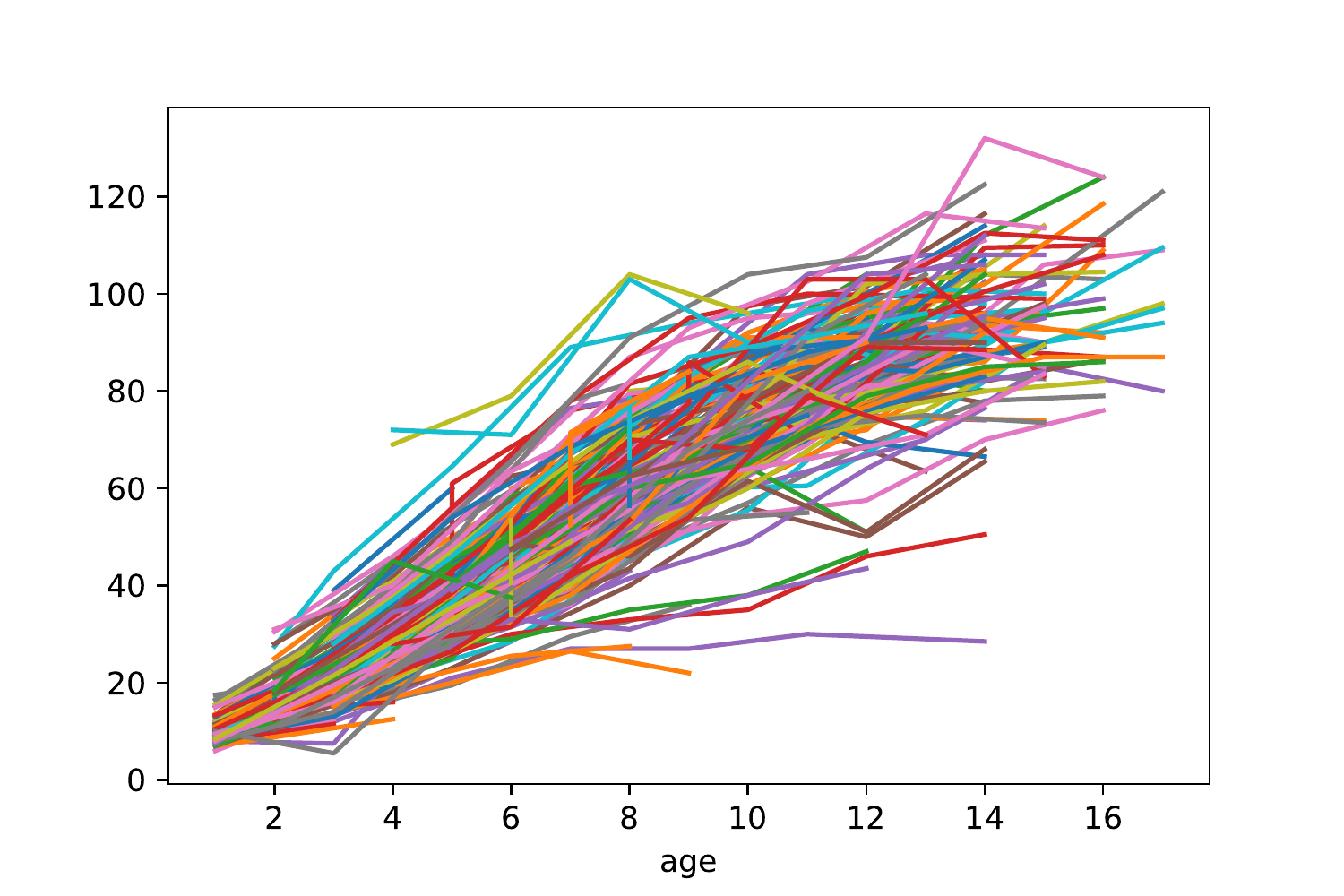}
	\caption{Evolution of body weight (in g) as a function of age (in days from hatching)}
	\label{fig:coucaldata}
\end{figure}

The logistic growth mixed model defined in Section \ref{sec:app_logistic_model} was fitted to the dataset, with the asymptotic weight and the inflexion point (i.e. the age at which bird $i$ reaches half its asymptotic body mass) as random effects. The tuning parameters of the algorithm were set as follows: $K_{pre-heating} = 2000$, $K=10000$, $C_{heating} = 100$, $\alpha =2/3$, $\lambda_0 = 10^{-4}$, and the algorithm was initialized at a random value for $\theta$. Results are given in Figure \ref{fig:evolthetacoucal}. The final estimates were $\hat{\beta}_1 = 97.15$,  $\hat{\beta}_2 = 6.50$, $\hat{\alpha} = 2.80$, $\hat{\Gamma}_{11} = 271.00$, $\hat{\Gamma}_{12} = 7.75$, $\hat{\Gamma}_{22} = 1.10$ and $\hat{\sigma}^2 = 19.80$. Our results are consistent with those provided by the \texttt{semix} package implemented in \texttt{R}, which performs maximum likelihood estimation using the SAEM algorithm for which theoretical guarantees only exist in the exponential family setting. However in our case, due to the presence of a fixed effect, the model does not belong to the curved exponential family as explained in the main body of the paper.

\begin{figure}
	\centering
	\includegraphics[scale=0.65]{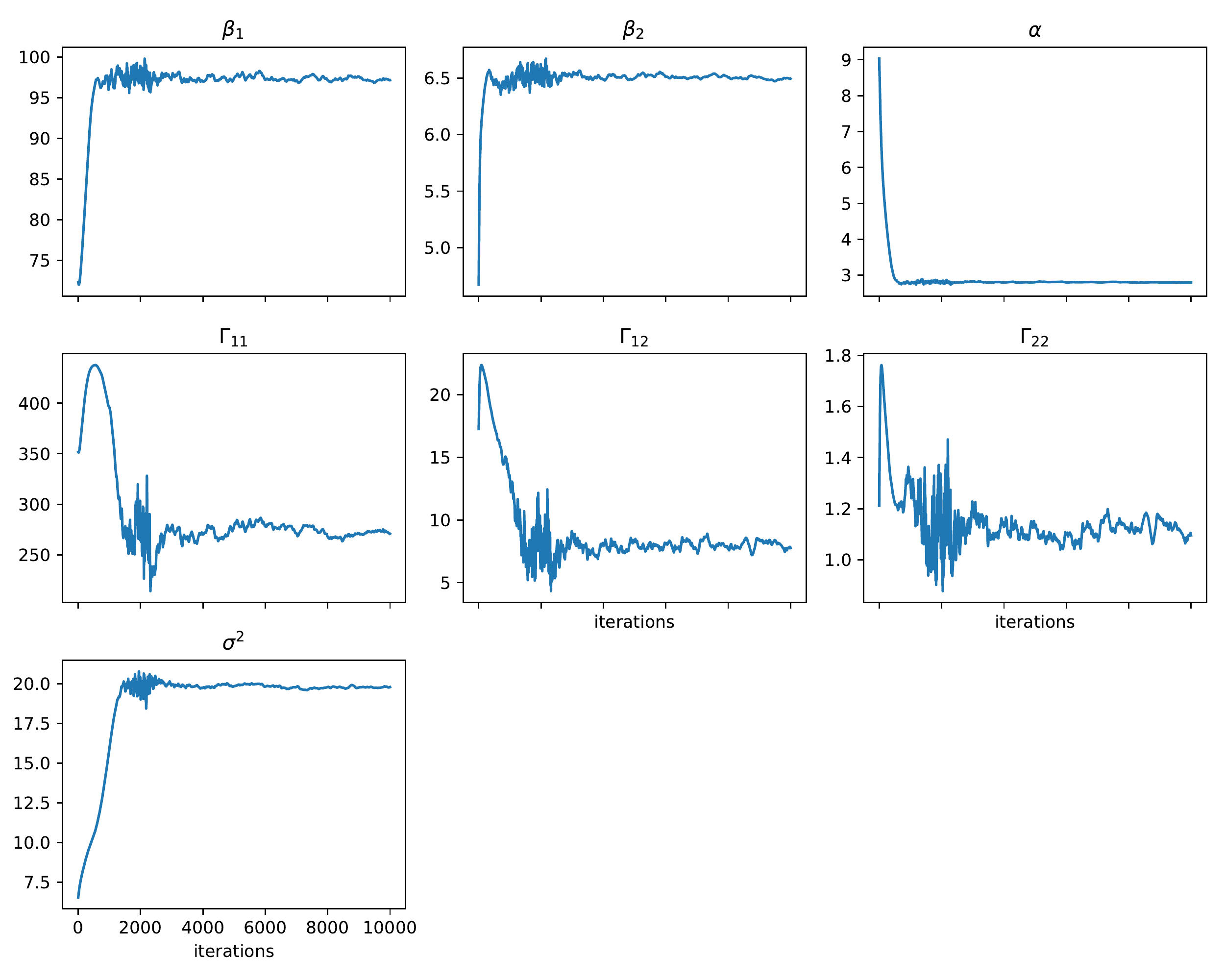}
	\caption{Evolution of the parameter estimates across the iterations on the real dataset.}
	\label{fig:evolthetacoucal}
\end{figure}

\FloatBarrier

\clearpage
\subsection{Stochastic Block Model}\label{sec:app:SBM}
\begin{figure}[ht]
	\centering
	\includegraphics[width=.7\textwidth]{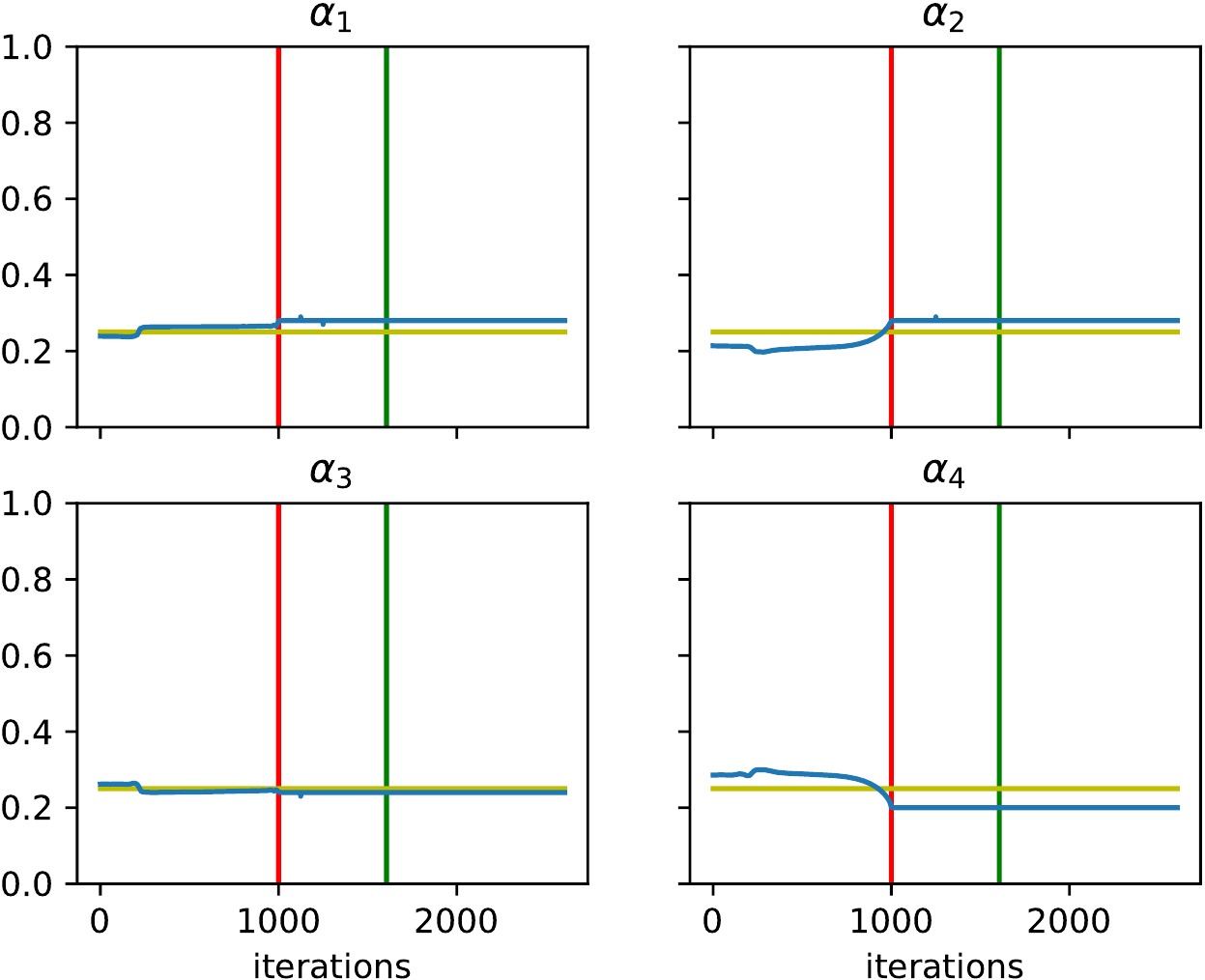}
	\caption{Evolution of the $\alpha$ estimates across the iterations with $\nobs=100$ and $K=4$. Yellow line: simulated value. The red line is the end of the pre-heating, and the green line is the end of the heating.}
	\label{fig:sbm_all_alpha}
\end{figure}

\begin{figure}[ht]
	\centering
	\includegraphics[width=.95\textwidth]{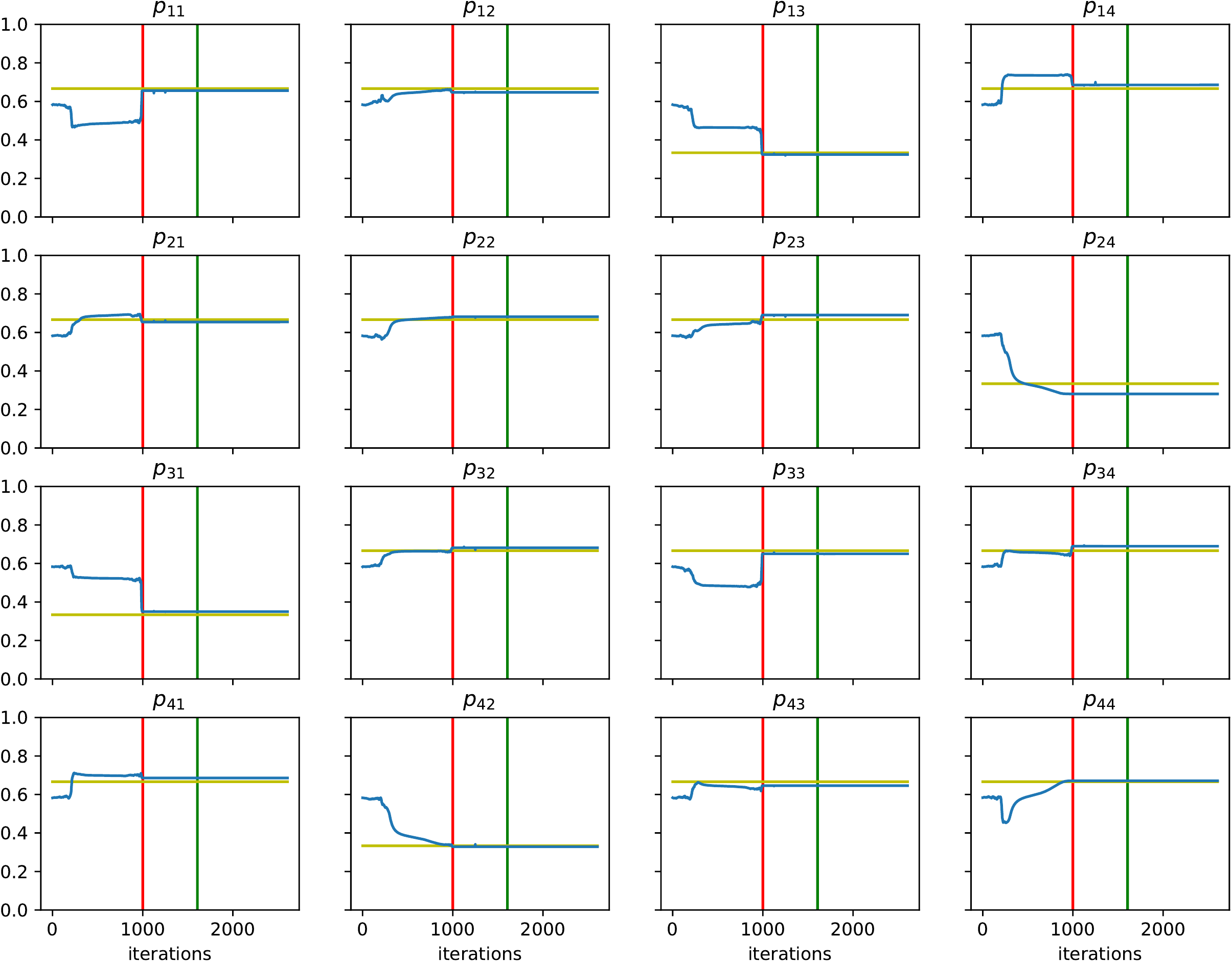}
	\caption{Evolution of the $p$ estimates across the iterations with $\nobs=100$ and $K=4$. Yellow line: simulated value. The red line is the end of the pre-heating, and the green line is the end of the heating.}
	\label{fig:sbm_all_pi}
\end{figure}

\begin{table}[ht]
	\centering
	\begin{tabular}{c|c|cc|cc}
		&& \multicolumn{2}{c|}{$\nobs=100$} & \multicolumn{2}{c}{$\nobs=200$} \\
		Parameter & Simulated value & RMSE & Empirical coverage & RMSE & Empirical coverage \\
		\hline
		Global $\theta$ &         & $0.648$ & $0.936\pm0.011$& $0.415$ & $0.956\pm0.009$ \\
		$\alpha_1$      & $0.250$ & $0.044$ & $0.943\pm0.010$& $0.031$ & $0.940\pm0.010$ \\
		$\alpha_2$      & $0.250$ & $0.044$ & $0.939\pm0.010$& $0.031$ & $0.943\pm0.010$ \\
		$\alpha_3$      & $0.250$ & $0.044$ & $0.941\pm0.010$& $0.031$ & $0.931\pm0.011$ \\
		$\alpha_4$      & $0.250$ & $0.044$ & $0.945\pm0.010$& $0.031$ & $0.940\pm0.010$ \\
		$p_{1,1}$       & $0.667$ & $0.023$ & $0.940\pm0.010$& $0.012$ & $0.949\pm0.010$ \\
		$p_{1,2}$       & $0.667$ & $0.019$ & $0.947\pm0.010$& $0.010$ & $0.949\pm0.010$ \\
		$p_{1,3}$       & $0.333$ & $0.022$ & $0.948\pm0.010$& $0.012$ & $0.953\pm0.009$ \\
		$p_{1,4}$       & $0.667$ & $0.019$ & $0.947\pm0.010$& $0.010$ & $0.948\pm0.010$ \\
		$p_{2,1}$       & $0.667$ & $0.020$ & $0.944\pm0.010$& $0.010$ & $0.947\pm0.010$ \\
		$p_{2,2}$       & $0.667$ & $0.022$ & $0.945\pm0.010$& $0.012$ & $0.952\pm0.009$ \\
		$p_{2,3}$       & $0.667$ & $0.020$ & $0.940\pm0.010$& $0.010$ & $0.953\pm0.009$ \\
		$p_{2,4}$       & $0.333$ & $0.020$ & $0.942\pm0.010$& $0.011$ & $0.951\pm0.010$ \\
		$p_{3,1}$       & $0.333$ & $0.023$ & $0.943\pm0.010$& $0.013$ & $0.943\pm0.010$ \\
		$p_{3,2}$       & $0.667$ & $0.019$ & $0.950\pm0.010$& $0.010$ & $0.939\pm0.011$ \\
		$p_{3,3}$       & $0.667$ & $0.023$ & $0.941\pm0.010$& $0.013$ & $0.955\pm0.009$ \\
		$p_{3,4}$       & $0.667$ & $0.020$ & $0.943\pm0.010$& $0.009$ & $0.950\pm0.010$ \\
		$p_{4,1}$       & $0.667$ & $0.020$ & $0.939\pm0.010$& $0.009$ & $0.957\pm0.009$ \\
		$p_{4,2}$       & $0.333$ & $0.020$ & $0.949\pm0.010$& $0.011$ & $0.947\pm0.010$ \\
		$p_{4,3}$       & $0.667$ & $0.019$ & $0.955\pm0.009$& $0.010$ & $0.952\pm0.009$ \\
		$p_{4,4}$       & $0.667$ & $0.024$ & $0.946\pm0.010$& $0.011$ & $0.954\pm0.009$ \\
	\end{tabular}
	\caption{Detailed result for numerical experiments on SBM with $2000$ replications. The RMSE given for $\para$ is the empirical mean of ${\|\widehat\theta-\theta_0\|}^2$ and the coverage is the coverage of the confidence ellipsoid in $\mathbb R^{Q^2+Q-1}$ built at the nominal level of 0.95. For all original parameter, the RMSE is computed after transformation of $\widehat\theta$ in original parameter space, and the confidence interval built at the nominal level of 0.95 is obtained by applying delta method with the Fisher Information Matrix estimate.}
	\label{tab:emp_sbm_details}
\end{table}

\FloatBarrier

\end{document}